\documentclass[sn-mathphys]{sn-jnl}
\usepackage{url}
\usepackage{float}
\usepackage{lmodern}

\jyear{2024}%
\usepackage{graphicx}
\usepackage{epstopdf}
\usepackage{moreverb}
\usepackage{algorithm}
\usepackage{algpseudocode}
\usepackage{footmisc}
\usepackage{caption}
\usepackage{subfig,graphicx}
\usepackage{multirow}
\usepackage{bbm}
\usepackage{lipsum} 

\DeclareMathOperator{\pr}{\mathbb P}
\DeclareMathOperator{\E}{\mathbb E}

\DeclareMathOperator{\ind}{\mathbb I}

\DeclareMathOperator*{\argmin}{arg\,min}

\newcommand{\Real}{\mathbb R}

\newcommand{\NatInt}{\mathbb N}

\newcommand{\CalF}{\mathcal F}

\newcommand{\CalH}{\mathcal H}

\newcommand{\CalR}{\mathcal R}

\newcommand{\CalL}{\mathcal L}

\newcommand{\CalO}{\mathcal O}
\newcommand{\CalG}{\mathcal G}
\newcommand{\CalS}{\mathcal S}
\newcommand{\CalT}{\mathcal T}

\newcommand{\rmI}{\boldsymbol{\mathrm I}}
\newcommand{\BFx}{\boldsymbol{x}}
\newcommand{\BFX}{\boldsymbol{X}}

\newcommand{\BFP}{\boldsymbol{P}}
\newcommand{\BFb}{\boldsymbol{b}}
\newcommand{\BFW}{\boldsymbol{W}}

\newcommand{\BFU}{\boldsymbol{U}}
\newcommand{\BFv}{\boldsymbol{v}}
\newcommand{\BFM}{\boldsymbol{M}}

\newcommand{|}{\vert}

\newcommand{\BFS}{\boldsymbol{S}}

\theoremstyle{thmstyleone}%
\newtheorem{theorem}{Theorem}
\newtheorem{proposition}{Proposition}
\newtheorem{assumption}{Assumption}
\newtheorem{lemma}{Lemma}

\theoremstyle{thmstyletwo}%
\newtheorem{example}{Example}%
\newtheorem{remark}{Remark}%

\theoremstyle{thmstylethree}%
\newtheorem{definition}{Definition}%

\begin{document}
\title[Improved Convergence Rate of Nested Simulation with LSE on Sieve]{Improved Convergence Rate of Nested Simulation with LSE on Sieve \footnote[1]{ 
Corresponding author: Lu Zou, \url{luzou0330@szpu.edu.cn}\\ This work  is partly supported by National Natural Science Foundation of China (No.72301076, No.12101149) and Science and Technology Commission of Shanghai Municipality (No.23PJ1400800).}
}


 \author[1]{\fnm{Ruo-xue} \sur{Liu}\footnote[2]{Division of Emerging Interdisciplinary Areas, Academy of Interdisciplinary Studies, The Hong Kong University of Science and Technology, Clear Water Bay, Kowloon, 999077, Hong Kong. }}

 \author[2]{\fnm{Liang} \sur{Ding}\footnote[3]{Schoold of Data Science, Fudan University, 220 Handan Rd, Yangpu District, Shanghai, 200437, China. }}

 \author[3]{\fnm{Wen-jia} \sur{Wang}\footnote[4]{\label{foot:hkust}Data Science and Analytics Thrust, Information Hub, The Hong Kong University of Science and Technology (Guangzhou), 1 Duxue Rd, Nansha District, Guangzhou, Guangdong, 511453, China. }}

 \author[4]{\fnm{Lu} \sur{Zou}\footnote[1]{\label{foot:}School of Management, Shenzhen Polytechnic University, 7098 Liuxian Avenue, Nanshan District, Shenzhen, Guangdong, 518055, China.}}




\abstract{ Nested simulation encompasses the estimation of functionals linked to conditional expectations through simulation techniques. In this paper, we treat conditional expectation as a function of the multidimensional conditioning variable and provide asymptotic analyses of general non-parametric Least Squared Estimators on sieve, without imposing specific assumptions on the function's form. Our study explores scenarios in which the convergence rate surpasses that of the standard Monte Carlo method and the one recently proposed based on kernel ridge regression. We use kernel ridge regression with inducing points and neural networks as examples to illustrate our theorems. Numerical experiments are conducted to support our statements. }



\maketitle

\section{Introduction}\label{sec:intro}

Simulation often aims to estimate a function linked to conditional expectations, which can take the form of either an expectation or a quantile. Typically, there is no analytical expression readily available for either the conditional expectation or the function itself. Therefore, the estimation process involves two simulation layers.  First, in the ``outer'' level, we simulate the random variable being conditioned on. Then, in the ``inner'' level, we simulate another random entity, conditioning on each previously generated sample (referred to as a ``scenario''). This problem category is known as \emph{nested simulation}. In this paper, our primary focus centers on using simulation to estimate the following quantity:
\begin{align}\label{eqgoal}
    \theta = \mathcal{T}(\mathit{E}[Y \vert  X]),
\end{align}
where $X$ is a random variable in $\mathbb{R}^d$, $Y$ is a real-valued random variable, and $\mathcal{T}$ is a functional mapping of a probability distribution to a real number.
Nested simulation is closely related to conditional Monte Carlo, a versatile technique for reducing variance in simulations \citep[Chapter 5]{AsmussenGlynn07}. Any expectation, denoted as $\mathit{E}[Y]$, can be expressed as $\mathit{E}[\mathit{E}[Y \vert X]]$. By the law of total variance, the variance of $\mathit{E}[Y \vert  X]$ is less than or equal to the variance of $Y$. Consequently, $\mathit{E}[\mathit{E}[Y \vert X]]$ serves as an unbiased estimator with lower variance than $Y$. It's especially valuable when $X$ and $Y$ are strongly correlated, and when the conditional expectation can be computed or estimated efficiently. Additionally, conditional Monte Carlo can be used as a smoothing technique for gradient estimation \citep{FuHu97,FuHongHu09}. We highlight the motivation for nested simulation through its application in quantifying input uncertainty within stochastic simulations\citep{barton2014quantifying,xie2014bayesian}.  Consider a complex queue system driven by customer arrivals. The decision-maker needs to estimate the average waiting time of the customers, which can be expressed as $\mathit{E}[Y \vert X]$. In this example, the input $X$ can be represented as the arrival rates of customers at different times of the day. Generally, we only have access to a finite number of samples of customer arrival times, from which the input density of $X$ is estimated. The uncertainty associated with the estimated distribution of the input $X$ frequently significantly influences the accuracy of the estimated performance of the system.

An important research topic in nested simulation is the allocation of simulation resources, specifically the determination of outer-level scenarios and inner-level samples.  In this study, we use equal inner-level samples per outer-level scenario, a standard
treatment in the literature\citep{GordyJuneja10,BroadieDuMoallemi15}. For a given budget $\Gamma$, standard nested simulation has outer-level sample size $n$ growing as $\CalO(\Gamma^{2/3})$ and inner-level sample size $m$ for each out-level sample as $\CalO(\Gamma^{1/3})$. Under this rule of allocating the simulation budget achieves a convergence rate of the root mean squared error (RMSE) at $\CalO(\Gamma^{-1/3})$ \citep{GordyJuneja10,ZhangLiuWang22}, which is slower than the optimal rate $\CalO(\Gamma^{-1/2})$ over all Monte Carlo methods. 

Recent studies \citep{hong2017kernel, wang2024smooth} have highlighted opportunities to improve the optimal convergence rate $\CalO(\Gamma^{-1/3})$. Specifically, \cite{hong2017kernel} demonstrated faster convergence rates in terms of RMSE, while \citep{wang2024smooth} achieved faster convergence rates in terms of mean absolute error (MAE). In particular, if the conditional expectation function $f(\BFx) = \mathit{E}[Y \vert X = \BFx]$ is a \emph{Sobolev} function with respect to $\BFx$, a convergence rate for the MAE of $\CalO(\Gamma^{-1/2})$ can be achieved under certain conditions \citep{wang2024smooth}. This condition allows for the exchange of information among different outer-level scenarios ${\BFx}$ and can be effectively estimated using \emph{kernel ridge regression }(KRR) \citep{Hastie20}. In this study, we extend the settings in \cite{wang2024smooth} to more general function spaces and estimators. Specifically, we unify the traditional nested simulation and the recent KRR method under the so-called least squared estimation (LSE) on sieve. Within this framework, we have the flexibility to incorporate various advanced estimation techniques, such as KRR with inducing points and neural networks, to enhance the overall effectiveness of nested simulation. We then investigate the circumstances under which the best possible convergence rate $\CalO(\Gamma^{-1/2})$ in terms of $l_p$ norm is attainable. Our findings encompass convergence results for both the MAE and the RMSE.

\subsection{Related Literature}

In recent years, the application of nonparametric machine learning methods in nested simulation has attracted considerable interest. Key smoothing techniques include stochastic kriging, kernel smoothing, and likelihood ratio approaches, as detailed in \citep{broadie2015risk,hong2017kernel,hong2009estimating,barton2014quantifying,lin2020fast,liu2010stochastic}. These methods aim to enhance the efficiency of nested estimation by employing nonparametric machine learning to approximate conditional expectations. Parametric approaches such as the linear basis function model remain popular for their use of flexible basis functions, including polynomials, splines, and radial basis functions. Customizing these basis functions to fit specific domains can further improve performance. For instance, \citep{broadie2015risk} utilized this model in portfolio risk management under nested simulation, combining polynomial risk factors with payoff-based basis functions. Similarly, \citep{LinYang20} employed spline regression techniques to estimate insurance portfolio risks. In addition, non-parametric techniques like stochastic kriging (a form of Gaussian process regression) have been effectively employed in nested simulation. For example, \citep{liu2010stochastic} applied stochastic kriging to conditional value-at-risk estimation. Such methods improve sample efficiency in nested simulations by enabling information exchange across scenarios at the outer level. Likewise, the likelihood ratio method achieves comparable benefits by utilizing likelihood ratios to efficiently use all inner-level data for conditional expectation estimation, as demonstrated in \citep{feng2020optimal,liu2019online,zhang2022sample}.

Another area within the field of nested simulation aims at deriving budget allocation policies to minimize the asymptotic MSE. \cite{GordyJuneja10} showed that minimizing the asymptotic MSE through budget allocation yields $m=c\Gamma^{1/3}+o(\Gamma)$, and $b=c^{-1}\Gamma^{2/3}+o(\Gamma)$, where $c$ is an unknown constant. The effective application of nested simulation is challenging because the unknown constant $c$ significantly impacts the allocation scheme, as well as the MSE of the nested simulation estimators. Recently,  \cite{zhang2022bootstrap} introduced a novel sample-driven budget allocation rule that utilizes the bootstrap approach and LSM to estimate the constant $c$ showcasing its asymptotic effectiveness. Building on the quest for efficiency in variance reduction, \cite{liu2024kernel} introduced a kernel quantile estimation for variance reduction, specifically addressing the risk measure. Additionally, \cite{liang2024fast} developed a new estimator that employs the jackknife technique for estimating $ \theta =\mathit{E}[ \mathcal{T}(\mathit{E}[Y \vert  X])]$.  Beyond point estimation, \cite{lan2010confidence}  introduced a novel two-level simulation procedure that constructs confidence intervals (CIs) for the variance of conditional expectation. More recently, \cite{cheng2022constructing} established unified CIs for $\mathcal{T}(\mathit{E}[Y | X])$, achieving optimal convergence rates of $\mathcal{O}(\Gamma^{-1/3})$ for the CIs width in nested simulations involving various functional forms. For further literature on CIs in nested simulation, please see \citep{cheng2021non,ZhuLiuZhou20}.


\subsection{Notation and Organization}
Throughout the paper, we use the following notation. For independent and identically distributed (i.i.d.) outer scenario $\{\BFx_i\}_{i=1}^n$, let $P_n=\frac{1}{n}\sum_{i=1}^n\delta_{\BFx_i}$ be the empirical density of $X$ with $\delta_{\BFx_i}$ a point mass at $\BFx_i$. For any functions $f$, $g$ defined on $\{\BFx_i\}_{i=1}^n$, define
the empirical inner product of $f$ and $g$ as $\langle f, g\rangle_n=\int f(\BFx)g(\BFx)\mathrm{d}P_n$ 
and the empirical $L_2$ norm $\|f\|_n^2=\langle f,f\rangle_n.$ We also define the $L_2$ norm induced by the distribution of $X$, denoted by $P_X$, as $\|f\|_{L_2(X)}=\int f^2\mathrm{d}P_X$. Let $\Omega$ be the support of the distribution of $X$. Then the $L_\infty$ norm of a function $f$ is defined as $\|f\|_{L_\infty}=\sup_{\BFx\in\Omega}|f(\BFx)|$. For any integer $K$, let $[K]$ denote the set $\{1,\cdots,K\}.$

For two positive sequences $a_n$ and $b_n$, $a_n = \CalO_{\pr}(b_n)$ means that for any $\varepsilon>0$, there exists $C>0$ such that
$\pr(a_n >C b_n) <\varepsilon$
for all $n$ large enough.  We also write 
$a_n=\CalO(b_n)$ and $a_n\asymp b_n$
if there exist some constants $C,C^\prime >0$ such that 
$a_n\leqslant C b_n$ and 
$C^\prime \leqslant a_n/b_n \leqslant C $, respectively, for all $n$ large enough.

The remainder of the paper is organized as follows.
In Section~\ref{sec:problem}, we introduce the background of nested simulation and formulate the research question.
In Section~\ref{sec:LSE}, we study the general framework of LSE on sieve.
In Section~\ref{sec:asymptotic}, we analyze its asymptotic properties. 
In Section~\ref{sec:example}, we use KRR with inducing points and neural networks as examples.
In Section~\ref{sec:numerical}, we conduct numerical experiments to verify our theorems. 
In Section~\ref{sec:conclusion}, we conclude the paper. 
Technical results and proofs are presented in the Appendix.

\section{Problem Formulation}\label{sec:problem}
In this paper, we consider the following forms of functional $\CalT$ in \eqref{eqgoal}:
\begin{enumerate}
\item $\mathcal{T}(Z) = \mathit{E}[\eta(Z)]$ for some function $\eta:\Real^d\mapsto \Real$; \label{form:exp} 
\item $\mathcal{T}(Z) =  \inf\{z\in\Real:\pr(Z\leqslant z)\geqslant \tau\}$, i.e., the VaR of $Z$ at level $\tau\in(0,1)$. \label{form:risk}
\end{enumerate}
In our study, input $Z$ is the conditional random variable $\mathit{E}[Y|X]$. Functionals in the Form \ref{form:exp} are known as \emph{nested expectations} and are employed in supervised learning, including stochastic kriging and classification. On the other hand, functionals represented by Form \ref{form:risk} are linked to \emph{risk measures} and find applications in risk management.

The nested simulation usually involves the following three steps to generate data. 
(i) Generate $n$ i.i.d.  outer-level scenarios $\{\BFx_i:i=1,\cdots,n\}$. 
(ii) For each $\BFx_i$, generate $m$ i.i.d. inner-level samples from the conditional distribution of $Y$ given $X=\BFx_i$, denoted by $\{y_{ij}:j=1,\cdots,m\}$, where $y_{ij}$ is modeled as a true function of the conditional expectation $f(\BFx_i)=\mathit{E}[Y\vert \BFx_i]$ plus an observation noise at j-th inner-level, see assumption \ref{assump:subG} in section \ref{sec:LSE} . 
(iii) An estimator $\hat{f}(X)$ of $f(X)=\mathit{E}[Y\vert X]$, adapted to the samples $(X, Y)$, is constructed to estimate the  target $\theta$.

Let $\hat{f}_{(1)} \leqslant \cdots \leqslant \hat{f}_{(n)}$  denote the order statistics of $\{\hat{f}(\BFx_1),\cdots, \hat{f}(\BFx_n)\}$. Then, $\hat{f}_{(\lceil \tau n\rceil)}$ is the sample quantile at level $\tau$, 
where $\lceil z\rceil$ denotes the least integer greater than or equal to $z$.  
Nested simulation then estimates $\theta$ via 
\begin{equation}\label{eq:standard-estimator}
\hat\theta_{n,m} = 
\left\{
\begin{array}{ll}
n^{-1}\sum_{i=1}^n \eta(\hat{f}(\BFx_i) ),     &\quad\mbox{if } \mathcal{T}(\cdot)= \E[\eta(\cdot)],  \\[0.5ex]
\hat{f}_{(\lceil \tau n\rceil)},     &\quad  \mbox{if }\mathcal{T}(\cdot) = \mbox{quantile}.
\end{array}
\right.
\end{equation}

A central problem in nested simulation is budget allocation. Given a simulation budget $\Gamma$ and an estimator $\hat\theta_{n,m}$, the optimal values for $n$ and $m$ to minimize the error in estimating $\theta$ can be determined by analyzing the asymptotic behavior of the statistics $|\hat\theta_{n,m} - \theta|$. \cite{GordyJuneja10} and \cite{ZhangLiuWang22} demonstrate that, without any assumption on $f$, the simulation budget should be allocated in such a way that $n\asymp \Gamma^{2/3}$ and $m\asymp \Gamma^{1/3}$ as $\Gamma\to\infty$, resulting in 
\[
\left(\mathit{E}\left[(\hat\theta_{n,m}-\theta)^2\right]\right)^{1/2} \asymp \Gamma^{-1/3}.
\] 
\cite{wang2024smooth} improves upon these results by imposing a smoothness condition on $f$ and shows that, for $f$ smooth enough, the simulation budget can be reduced to $n\asymp \Gamma$ and $m\asymp 1$ as $\Gamma\to\infty$, leading to
\[
|\hat\theta_{n,m}-\theta| = \mathcal{O}_p(\Gamma^{-1/2}).
\]

\section{LSE on Sieve}\label{sec:LSE}
In standard nested simulations, the estimator $\hat{f}(\BFx_i)$ can be calculated as the sample averages by $\bar{y}(\BFx_i) = \frac{1}{m}\sum_{j=1}^m y_{ij}$. Alternatively, it can be expressed as the following least squared estimator:
\begin{equation}\label{eq:sample_avg}
\hat{f} = \argmin_{h\in\mathcal{F}_n}\|h-\bar{y}\|_n^2
\end{equation}
where $\mathcal{F}_n$ is defined as $\{\ind_{\{\BFx_i=\BFx\}}:i=1,\cdots,n\}$. On the other hand, \cite{wang2024smooth} constructs $\hat{f}$ using KRR, which can be seen as the following least squared estimator:
\begin{equation}\label{eq:KRR}
\hat{f} = \argmin_{h\in\mathcal{H}(\delta_n)}\|h-\bar{y}\|_n^2
\end{equation}
where $\mathcal{H}(\delta)$ is defined as $\{h\in\mathcal{H}: \|h\|_{\mathcal{H}}\leqslant \delta\}$, $\|\cdot\|_{\mathcal{H}}$ denotes the norm of a Sobolev space $\mathcal{H}$, and $\delta_n>0$ is adaptive to sample sizes $n,m$; see  \citep{KanHenSejSri18,adams2003sobolev} for details.

Both \eqref{eq:sample_avg} and \eqref{eq:KRR} are evidently instances of LSE on sieves. To see this, we begin by introducing the definition of LSE on sieve. Suppose the true function $\mathit{E}[Y|X=\BFx]$ resides in a function space $\CalF$. Consider another space of functions, denoted as $\mathcal{F}_{n,m}$, which is adaptable to outer-level scenarios $n$ and inner-level samples $m$. This space does not necessarily need to be a subspace of \(\mathcal{F}\), but it should not exceed \(\mathcal{F}\) in complexity. The concept of complexity will be discussed in Remark \ref{remark:donsker} following the introduction of the concepts of covering number and entropy. Define
\begin{equation}\label{eq:sieve}
    \hat{f}_{n,m}=\argmin_{h\in\mathcal{F}_{n,m}}\|h-\bar{y}\|_n^2.
\end{equation}
We say that $\CalF_{n,m}$ is
a sieve and $\hat{f}_{n,m}$ is an LSE on sieve. The sieve $\CalF_{n,m}$ can be a function space, a vector space, or a neural network accordingly.

Similar to \cite{wang2024smooth}, we impose the following assumptions on the data. 

\begin{assumption}\label{assump:support}
The support of $X$ is a bounded convex set  $\Omega\subset \Real^d$.
Moreover, $X$ has a probability density function that is bounded above and below away from zero.
\end{assumption}

\begin{assumption}\label{assump:subG}
 $y_{ij}$ conditioned on $\BFx_i$ is of the form
$$y_{ij} = f({\BFx}_i) + \varepsilon_{ij}$$
where $\varepsilon_{ij}$ represents independent zero-mean  random variables, which may not have identical distributions but satisfy the following $\sigma^2$-sub-Gaussian assumption:
\[\max_{i,j} \E e^{t\varepsilon_{ij}}\leqslant e^{\sigma^2t^2/2},\quad \forall t\in \Real.\]
\end{assumption}
A sub-Gaussian random variable represents a type of probability distribution characterized by exponential tail decay, ensuring that its tails are at least as contained as those of a Gaussian distribution. This property makes sub-Gaussian distributions particularly useful in modeling scenarios with controlled variability. Sub-Gaussian random variables are widely applicable across various fields. For instance, in our earlier customer arrival example, the inter-arrival times can often be considered bounded in real-world applications. Consequently, the queuing system can be effectively modeled using sub-Gaussian random variables.



In contrast to \cite{wang2024smooth}, which relies on specific assumptions about the underlying function $f$ and employs the KRR technique, our study adopts a more generalized framework. We investigate the necessary conditions governing the relationship between the complexity of $\mathcal{F}_{n,m}$ and the underlying function $f$, where the complexity of a space is quantified by the following measures.
\begin{definition}
    Let $\mathcal{M}$ be a space equipped with metric $\|\cdot\|$, and let $\vee$ denote the maximum of two values. For example, $a\vee b=\max(a,b)$. The \emph{$u$-covering number} $N(u,\mathcal{M},\|\cdot\|)$ is defined as the minimum number of balls with radius $u$ necessary to cover $\mathcal{M}$. The \emph{$u$-entropy} of $\mathcal{M}$ is then $H(u,\mathcal{M},\|\cdot\|)=\log N(u,\mathcal{M},\|\cdot\|) \vee 0$.
\end{definition}
\begin{remark} \label{remark:cover_num}
The \emph{$u$-entropy} of $\mathcal{M}$ determined by the minimum number of small subsets is needed to cover the space, which represents the complexity of space; see Figure \ref{fig:cover_num} for illustration. The \emph{$u$-entropy}  is pivotal in evaluating the complexity of function spaces and the capacity of certain models, where high metric entropy suggests a need for more data to accurately capture the underlying structures.
\end{remark}

\begin{figure}[htbp]
    \centering
    \includegraphics[scale=0.15]{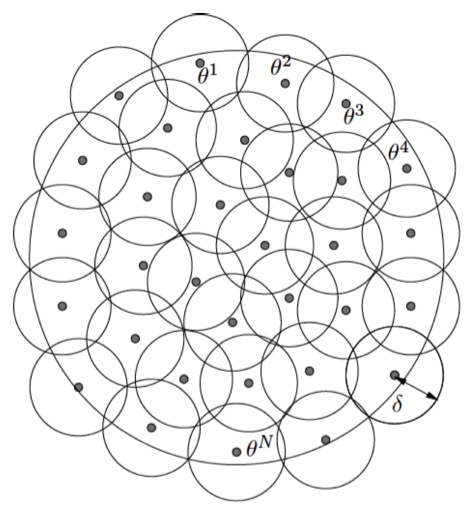}
    \caption{Illustration of Covering Sets. A \emph{$u$-covering} of $\mathcal{M}$ is a collection of element $\{\theta_1, \cdots, \theta_N\}$, such that for each $\theta\in \mathcal{M}$, there exists a $j\in\{1,\cdots, N\}$, such that $\|\theta-\theta_j\|\leqslant \delta$. (Figure 5.1 in \cite{Wainwright19}. )}\label{fig:cover_num}
\end{figure}

Following the above definition, we impose an additional assumption regarding the properties of the spaces $\CalF_{n,m}$ and $\CalF$. This assumption is mild and can be satisfied by both  $\CalF_n$ in \eqref{eq:sample_avg} and $\CalH(\delta_n)$ in \eqref{eq:KRR}.
\begin{assumption}\label{assump:b_bounded}
    (i) $H(u,\CalF_{n,m},\|\cdot\|_n)\leqslant H(u,\CalF,\|\cdot\|_n)<\infty$ for any $u>0$; (ii) there exists a constant $b<\infty$ such that $\sup_{h\in\CalF\cup\CalF_{n,m}}\|h\|_{L_\infty}<b$; (iii) $\CalF$ and $\CalF_{n,m}$ are Donsker with respect to the distribution of $X$ (see Definition \ref{def:donsker} in Appendix \ref{append:back_know} for details).
\end{assumption}
\begin{remark} \label{remark:donsker}
The covering number of $\CalF_{n,m}$ is dependent on $n$ and $m$; larger values of $n$ and $m$ lead to an increased covering number. This implies that we need to expand the capacity of $\CalF_{n,m}$ with greater knowledge of the target function spaces. Intuitively, a class of functions is called a Donsker class if, when you draw a large number of samples from a population and calculate certain statistics for each sample, the distribution of these statistics converges to a Gaussian process. This convergence mirrors the Central Limit Theorem for sums of random variables, yet it extends to a class of functions rather than applying to a singular statistic.
\end{remark}
In Assumption \ref{assump:support},  we assume that the support of $X$ is convex for simplicity in our mathematical treatment. This assumption can be extended to any compact set that meets the interior cone condition and features a Lipschitz boundary. Furthermore, we can extend our results to $\Real^d$ by applying the truncation technique to light-tailed densities. See \cite[Remark 3]{kohler2009optimal} for an example of the truncation technique.

In Assumption \ref{assump:subG}, the sub-Gaussian assumption for the noise terms can be generalized to a sub-exponential distribution, by applying an analysis framework similar to that outlined in this section. For the technical details of dealing with sub-exponential distribution, please refer to Section 10.1 in \cite{geer2000empirical}.  For errors with a heavy-tail distribution, our analytical approach may still be applicable. See \cite{kuchibhotla2022least} for a discussion on addressing heteroscedastic and heavy-tailed errors.

In Assumption \ref{assump:b_bounded},  uniform boundedness assumption for $\CalF$ and $\CalF_{n,m}$ is crucial in establishing the convergence rate of $\hat{\theta}_{n,m}$. This boundedness assumption is commonly employed in nonparametric regression settings, as discussed in \citep{guntuboyina2018nonparametric,koltchinskii2011oracle,rakhlin2017empirical}. If $\CalF$ and $\CalF_{n,m}$ is a non-Donsker class, the theoretical analysis for the convergence rate of $\hat{\theta}_{n,m}$ becomes significantly more complex. Furthermore, in the case of a general non-Donsker class $\CalF$, the LSE can be sub-optimal, as indicated by \cite{kur2021minimal}. Therefore, our estimator could be sub-optimal when $\CalF$ belongs to a non-Donsker class.

Before we present the main theorem, let us first define some key terms. Let
\[f^*_{n,m}=\argmin_{h\in\CalF_{n,m}}\|h-f\|_n,\quad\text{and}\quad \hat{f}_{n,m}=\argmin_{h\in\CalF_{n,m}}\|h-\bar{y}\|_n.\]
For $0<{\delta}<2^7\sigma/\sqrt{m}$, we define
\begin{align*}
    &\CalF_n({\delta})=\{h\in\CalF_{n,m}: \|h-f^*_{n,m}\|_n\leqslant {\delta}\},\\
    &  \mathcal{J}({\delta})=\int_{\frac{\delta^2\sqrt{m}}{2^7{\sigma}}}^{{\delta}}\sqrt{H(u,\CalF_n({\delta}),\|\cdot\|_n)}\mathrm{d}u \vee {\delta}.
\end{align*}
Given the entropy of $\CalF_{n,m}$, we are ready
to present our first theorem regarding the mean squared error (MSE) of LSE on sieve.

\begin{theorem}\label{thm:LSE_sieve}
Suppose Assumptions~\ref{assump:support}-- \ref{assump:b_bounded} hold. Take any upper bound $\Psi({\delta})\geqslant \mathcal{J}({\delta})$ in  such a way that $\Psi({\delta})/{\delta}^2$ is a non-decreasing function for any $0<{\delta}<2^7{\sigma}/\sqrt{m}$. Then for a constant $C$ only depends on $\sigma$, and constant $c$ is a universal positive constant independent of $n$, $m$ and the \emph{$u$-entropy} of $\CalF_{n,m}$, and for
\begin{equation}\label{eq:thm1_condition}
    \sqrt{nm}\delta^2_{n,m}\geqslant c\Psi(\delta_{n,m}),\quad \forall n,m,
\end{equation}
we have
\begin{align}\label{eq:LSE_convergence}
   &\pr\left( \|\hat{f}_{n,m}-f^*_{n,m}\|_n\geqslant \delta_{n,m}\right) \leqslant {\pr}\left(\|\bar{\varepsilon}\|_n\geqslant \frac{\sigma}{\sqrt{m}}\right)+C\exp\left[-\frac{nm\delta_{n,m}^2}{C^2}\right],\nonumber\\
   &\|\hat{f}_{n,m}-f^*_{n,m}\|_n\leqslant \CalO_{\pr}(\delta_{n,m}+\|f^*_{n,m}-f\|_n).
\end{align}
\end{theorem}

Theorem \ref{thm:LSE_sieve} considers both the sample sizes $n$ of outer scenario $\{\BFx_i\}$ and $m$ of the inner scenarios $\{y_{ij}\}$. From this theorem, the MSE of the LSE can be decomposed into two parts: the estimation error $\mathcal{O}_p(\delta_{n,m})$  and the approximation error $\|f_{n,m}^*-f\|_n$. Specially, the $\delta_{m,n}$ corresponds to variance and the $\|f^*_{n,m}-f\|_n$ corresponds to bias of the estimator.  Fundamentally, for a function class $\CalF_{m,n}$ with high complexity, the associated probability distribution in Theorem \ref{thm:LSE_sieve} is likely to exhibit a heavier tail, leading to increased variance. Conversely, a function class $\mathcal{F}_{m,n}$ of high complexity is also capable of closely approximating the true underlying function, resulting in reduced bias. Hence, $\delta_{m,n}$ and $\|f^*-f\|_n$  encapsulate the concept of a bias-variance trade-off in non-parametric problem settings.

\section{Asymptotic Analysis}\label{sec:asymptotic} 
If the asymptotic behavior of \( |\hat{\theta}_{n,m} - \theta| \) with respect to the sample sizes \( n \) and \( m \) is established, the optimal allocation of the budget \( \Gamma = nm \) can be determined. The following theorem demonstrates high probability convergence as well as convergence in \( l^p \) for any \( p \geqslant 1 \). 

For the nested expectation functional $\CalT$, we have the following theorem:
 \begin{theorem}\label{thm:nested_convergence}
     Let $\mathcal{T}(\cdot) = \mathit{E}[\eta(\cdot)]$ and $\eta$ be a twice-differentiable function with bounded first- and second-order derivatives.
Suppose Assumptions \ref{assump:support}--\ref{assump:b_bounded} hold. Moreover, $m$ satisfies $\sqrt{m}(\delta_{n,m}+\|f_{n,m}^*-f\|_n)=\CalO(1)$.
Then,  \[|\hat \theta_{n,m}-\theta| = \CalO_{\pr}\left(n^{-\frac{1}{2}}+\delta_{n,m}^2+\|f^*_{n,m}-f\|_n+\hat{\CalL}\right),\]
and for any integer $1\leqslant p<\infty$,
 \[\left(\mathit{E}|\hat \theta_{n,m}-\theta|^p\right )^{1/p} = \CalO\left(n^{-\frac{1}{2}}+\delta_{n,m}^2+\|f^*_{n,m}-f\|_n+\hat{\CalL}\right),\]
where $\delta_{n,m}$ and $f^*_{n,m}$ are as defined in Theorem \ref{thm:LSE_sieve} and
\begin{align*}
 \hat{\CalL}=\left|\mathit{E}\left[\eta'(f(X))\right]\mathit{E}_{X,\{\varepsilon_i\}} [f^*(X)-\hat{f}(X)]\right|.
\end{align*}
 \end{theorem}
 It is worth noticing that $\hat{\CalL}$ depends on the linearity $\hat{f}_{n,m}$. For linear estimators, i.e., $\hat{f}_{n,m}=\sum_{i=1}^n\phi_i(\cdot;X)\bar{y}_i$,  $\hat{\CalL}=0$ because
 \begin{align*}
      &\mathit{E}_{X,\{\varepsilon_i\}} [f^*(X)-\hat{f}(X)]=\mathit{E}_{X,\{\varepsilon_i\}}\bigl[\sum_{i=1}^n\phi_i \left(f(\BFx_i)+\varepsilon_i\right)-\phi_if(\BFx_i)\bigr]=0.
 \end{align*}
 Linear estimators include many commonly used estimators, such as KRR and spline regression. For some non-linear estimator, such as  wavelet-threshold estimators \citep{donoho1995noising}, it can be bounded by the approximation error $\|f^*_{n,m}-f_{n,m}\|_{L_2}$. For highly non-linear estimators such as neural networks, the analysis of $\hat{\CalL}$ becomes more complicated. The convergence behavior of neural networks is contingent upon diverse assumptions and architectures. Notably, for two-layer neural networks and infinitely wide neural networks, their equivalence to kernel methods has been elucidated through approaches like the Neural Tangent Kernel \cite{jacot2018neural}, enabling feasible analysis. Nonetheless, analyzing the behavior of other neural network architectures poses a greater challenge due to the complexities and diverse nature of their structures and functionalities.
 
 If $\CalT$ is the risk measure, we need to impose two more assumptions.
 \begin{assumption}\label{assump:L_infinity_embedding}
     (i) $\CalF$ and $\CalF_{n,m}$ are star-shaped, i.e., for any $h$ in $\CalF$ or $\CalF_{n,m}$, $ah$ is also in   $\CalF$ or $\CalF_{n,m}$ for all $0\leqslant a \leqslant 1$. (ii) for any $\delta >0$ and $h\in\CalF\cup\CalF_{n,m}$ with $\|h\|_{L_2(X)}\leqslant \delta$,  there exists universal constant $\alpha \in(0,1]$ such that $\|h\|_{L_\infty}\leqslant \CalO(\delta^\alpha)$. 
 \end{assumption}

\begin{assumption}\label{assump:margin-strong}
There exist positive constants $C_1$, $C_2$, $t_0$, and $\beta\leqslant \gamma$ such that 
\begin{align*}
C_2 t^{\gamma} \leqslant \pr(|f(X)-z|\leqslant t)\leqslant C_1 t^{\beta},\quad 
\forall t\in(0,t_0],\;\forall z\in \{f(\BFx):\BFx\in\Omega \}.
\end{align*}
\end{assumption}

Assumption \ref{assump:L_infinity_embedding} is for handling the error between the order statistics $\hat{f}_{(\lceil \tau n\rceil)}$ and the true quantile of $f(X)$, where we require point-wise asymptotic property. It is worth noting that this assumption is mild and can be met when $\mathcal{F}$ is chosen from commonly encountered spaces. For example, when $\CalF$ is an order-$m$ Sobolev space, we have $\alpha=1-1/(2m)$ \citep[Lemma 10.9]{geer2000empirical}, \cite[Sec 2.5]{van2014uniform}; when $\CalF$ is the reproducing kernel Hilbert space (RKHS) generated by  Gaussian kernel, we have $\alpha=1$ \citep[Lemma E.4]{ding2024random}; when $\CalF$ is a ReLU neural network, we have $\alpha=d/(d+2)$ as we will show in Section \ref{sec: ReLU}. These relationships are often referred as function space embeddings in general.  Although estimating the parameter $\alpha$, indicative of the function space's smoothness, frequently poses a challenge. Recent theoretical research suggests that by utilizing adaptive procedures, KRR and neural networks can retain robust performance even when facing smoothness misspecification. For more information, refer to \citep{dicker2017kernel, blanchard2018optimal, cherief2020convergence}.

Assumption \ref{assump:margin-strong}  is similar to Assumption 5 of \cite{wang2024smooth}.  To delve deeper into this assumption, let's define $F(z):=\pr((f(X))\leqslant z)$ as the cumulative distribution function (CDF) of $f(X)$, and its inverse, $F^{-1}(q):=\inf\{z:F(z)\geqslant q\}$, as the quantile function. The VaR for a given risk level $\tau\in(0,1)$ is thus $\zeta_{\rm VaR}:=\text{VaR}(f(X))=F^{-1}(\tau)$, and its estimation $\hat{\zeta}:=\hat{f}_{\lceil \tau n\rceil}$ is based on an estimator $\hat{f}$ for $f$. The aim is to show the convergence of $\hat{\theta}_{n,m}$ to $\theta$ through the convergence of $\hat{\zeta}$ to $\zeta_{\rm VaR}$. In the context of Assumption \ref{assump:margin-strong}, it is pertinent to note that $\pr(|f(X)-z|\leqslant t)=F(z+t)-F(z-t)$, and inherently, $F$ is a non-decreasing function. Therefore, Assumption \ref{assump:margin-strong} can be translated to:
\begin{align}
    &|F(z)-F(z')|\leqslant C_1|z-z'|^{\beta}, \label{eq:F_holder}\\
    &|F(z)-F(z')|\geqslant C_2|z-z'|^{\gamma}, \label{eq:F_inverse_holder}
\end{align}
for any $z,z'\in\{f(x):x\in\Omega\}$, with $C_1$ and $C_2$ being positive constants. Here, condition \ref{eq:F_holder} indicates that $F$ exhibits $\beta$-Hölder continuity, while condition \ref{eq:F_inverse_holder} signifies that $F^{-1}$ displays $\gamma^{-1}$-Hölder continuity. It's noteworthy that if $\beta \leqslant 1$, then a $\beta$-Hölder continuous function over an interval cannot be a constant. Given that neither $F$ nor $F^{-1}$ is inherently constant by definition, it logically follows from Assumption 5 that $\beta \leqslant 1 \leqslant \gamma$.

Equipped with these tools, we are now ready to present the following theorem.


\begin{theorem}\label{thm:risk_convergence}
    Let $\mathcal{T}(\cdot) = \inf\{z\in\Real:\pr(\cdot \leqslant z)\geqslant \tau\}$ for some $\tau\in(0,1)$. Suppose Assumptions~\ref{assump:support}--\ref{assump:margin-strong} hold. Then
    \[|\hat \theta_{n,m}-\theta|=\CalO_{\pr}\left(\left|\frac{\Delta_n}{\sqrt{m}}+\|f^*_{n,m}-f\|_n+\delta_{n,m}\right|^{\kappa}+n^{-\frac{1}{2\gamma}}\right), \]
    and for any $1\leqslant p<\infty$,
     \[\left(\mathit{E}|\hat \theta_{n,m}-\theta|^p\right )^{1/p} = \CalO\left(\left|\frac{\Delta_n}{\sqrt{m}}+\|f^*_{n,m}-f\|_n+\delta_{n,m}\right|^{\kappa}+n^{-\frac{1}{2\gamma}}\right),\]
where $\kappa=\alpha\beta/\gamma$, $\delta_{n,m}$ and $f^*_{n,m}$ are as defined in Theorem \ref{thm:LSE_sieve} and 
\begin{equation*}
    \Delta_n=\inf\left\{\delta: \frac{\delta^2}{b}\geqslant \inf_{\varepsilon\geqslant 0}\left\{4\varepsilon+\frac{C}{\sqrt{n}}\int_\varepsilon^{\delta} \sqrt{H(u,\CalH\cap B(\delta,\|\cdot\|_n),\|\cdot\|_n)}\mathrm{d}u\right\}\right\}.
\end{equation*}
\end{theorem}
Theorem \ref{thm:risk_convergence} elucidates that higher values of $\beta$ or lower values of $\gamma$ (equivalently, higher values of $\gamma^{-1}$) correlate with a more rapid convergence to the actual value $\theta$, particularly when considering the VaR as $\CalT$. As previously discussed, the parameters $\beta$ and $\gamma^{-1}$ represent the degree to which $F$ and $F^{-1}$ adhere to the Hölder condition, which in turn reflects their levels of smoothness. Therefore, for VaR, Theorem \ref{thm:risk_convergence} underscores how the estimator's effectiveness is influenced not just by the smoothness of $f(X)$'s cumulative distribution function (CDF) and its quantile function but also by the smoothness of $f$ itself and the dimensionality of the problem.

Given the implication from Assumption \ref{assump:margin-strong} that $\beta\leqslant 1\leqslant \gamma$, the optimal scenario for achieving the fastest convergence rate is when $\beta = \gamma = 1$. This condition implies that both $F$ and $F^{-1}$ exhibit Lipschitz continuity. Under these circumstances, the convergence rate posited in Theorem \ref{thm:risk_convergence} matches that of Theorem \ref{thm:nested_convergence}, assuming that the function class $\CalF$ and the sieve $\CalF_{n,m}$ fulfill Assumption \ref{assump:L_infinity_embedding} with $\alpha=1$. This observation highlights the intricate relationship between the smoothness of the CDF and quantile function of $f(X)$, and the convergence behavior of the estimator, thereby extending the known influences of function smoothness and dimensionality on estimator performance.
\begin{remark}
    Theorems \ref{thm:nested_convergence} and \ref{thm:risk_convergence} establish that our high-probability convergence also implies convergence in \( L^p \). To simplify notation, explicit \( L^p \) results are omitted, but the subsequent discussions remain applicable to \( L^p \) convergence.
\end{remark}

\section{Examples}\label{sec:example}
In this section, we show some examples to illustrate our theorems. These examples make it clear that the Least Squared Estimation (LSE) on sieve offers a higher degree of flexibility and, in certain specific scenarios, outperforms the KRR estimator. 

\subsection{Kernel Ridge Regression with Inducing Points }
In this subsection, we demonstrate that Kernel Ridge Regression (KRR) can be formulated as a Least Squares Estimator (LSE) on a sieve that is independent of the data. Moreover, we show that by selecting a data-adaptive sieve, we can construct an LSE on the sieve that is equivalent to KRR with inducing points. This estimator achieves an improved convergence rate. While inducing points may introduce some bias, they also reduce variance, and in many cases, the reduction in variance outweighs the increase in bias, potentially leading to a lower overall MSE and further improving convergence. Furthermore, as noted in \cite{burt2019rates}, a very small set of inducing points can be effectively utilized for large datasets without introducing significant bias in hyperparameter selection.

\subsubsection{Nested Expectation}\label{sec:KRR_nestedE}
We first consider the following KRR in constrained form with sieve $\CalF_{n,m}$  a fixed subset of the RKHS  generated by the Mat\'ern-$\nu$ kernel. This space is  \emph{norm-equivalent} to the $s=\nu+d/2$-order \emph{Sobolev space}  $\mathscr{H}^s(\Omega)$ equipped with norm $\|\cdot\|_{\mathscr{H}^s(\Omega)}$. Under our framework, we can analyze the asymptotic properties of the KRR estimator regardless of the specific form of the Mat\'ern kernel. More precisely, the KRR problem is equivalent to the following constrained optimization problem for $\nu>d/2$:
\begin{equation}\label{eq:constrained_KRR}
    \min_{h}\|h-\bar{y}\|_n^2,\quad\text{s.t.}\ \|h\|_{\mathscr{H}^s(\Omega)}^2\leqslant \CalR
\end{equation}
where  $\CalR\geqslant \|f\|^2_{\mathscr{H}^s(\Omega)}$ is independent of $n,m$. In this case, the approximation error $\|f^*-f\|_n=0$ because $f\in\CalF_{n,m}=\{h:\|h\|_{\mathscr{H}^s(\Omega)}^2\leqslant \CalR\}$. Furthermore, according to \cite{Birman_1967}, the $u$-entropy of $\mathscr{H}^s(\Omega)$ is $u^{-\frac{d}{s}}$, and according to \cite[Theorem 6.3]{geer2000empirical}, $\CalF=\CalF_{n,m}$ is Donsker. Therefore, it is straightforward to check that $\CalF=\CalF_{n,m}$ satisfy Assumptions~\ref{assump:b_bounded} and \ref{assump:L_infinity_embedding}. Applying Theorem \ref{thm:LSE_sieve} yields
\[\delta_{n,m}=\inf\{\delta>0: \sqrt{nm}\delta^2\geqslant c \delta ^{1-\frac{d}{2m}}\}=\CalO\left(|nm|^{-\frac{s}{2s+d}}\right).\]
When $\eta$ has bounded second derivative, we can use Theorem~\ref{thm:nested_convergence} to derive that 
\[|\hat \theta_{n,m}-\theta|=\Big|\frac{1}{n}
\sum_{i=1}^n\hat{f}\circ\eta(\BFx_i)-\mathit{E}[\eta\circ f(X)]\Big| = \CalO_{\pr}\left(n^{-\frac{1}{2}}+|nm|^{-\frac{2s}{2s+d}}\right).\]
By setting $m=\CalO(1)$ and $n=\CalO(\Gamma)$, it is straightforward to derive that we can have the improved convergence rate $|\hat \theta_{n,m}-\theta|=\CalO(\Gamma^{-\frac{1}{2}})$. This result is exactly the same as the improved convergence rate in  \cite{wang2024smooth} for $\nu>d/2$. 

However, for $\nu\leqslant d/2$, \cite{wang2024smooth} has sub-optimal result. We now show that by adaptively choosing the sieve $\CalF_{n,m}$, we can still have the optimal convergence rate $\Gamma^{-\frac{1}{2}}$. Without loss of generality, suppose $X$ is uniformly distributed on $\Omega$. We then select $\CalS_{n}=\CalO(n^{\frac{d}{2s+d}})$ points from $\{\BFx_i\}$, denoted as $\{\tilde{\BFx}_i\}$ such that the filled distance of $\tilde{\BFx}_j$ satisfy
\[h_{\Omega,\tilde{\BFX}}=\sup_{\BFx\in\Omega}\inf_{j\in[\CalS_n]}\|\tilde{\BFx}_j-\BFx\|=\CalO(\CalS_n^{-\frac{1}{d}}).\]
Let $\CalF_{n,m}=\text{span}\{k(\tilde{\BFx}_i,\cdot)\}$ where $k$ is the Mat\'ern-$\nu$ kernel. According to \cite{wu1993local}, 
\[\inf_{h\in\CalF_{n,m}}\|h-f\|_n=\CalO_{\pr}(\CalS_n^{-s/d})=\CalO_{\pr}(n^{-\frac{s}{2s+d}}).\]
Because $\CalF_{n,m}$ is a finite-dimensional space, we can use \cite[Corollary 2.5]{geer2000empirical} to derive that
\[H(u,\CalF_n(\delta),\|\cdot\|_n)\leqslant Cn^{\frac{d}{2s+d}}\log(\frac{4\delta+u}{u}).\]
From straightforward calculations, it can be shown that $\delta_{n,m}=\CalO(n^{-\frac{s}{2s+d}}$). By setting $n=\CalO(\Gamma)$ and $m=\CalO(1)$, together with the fact that $s=\nu+d/2> d/2$, we have the same convergence rate $|\hat \theta_{n,m}-\theta| =\CalO_{\pr}(\Gamma^{-\frac{1}{2}})$.

The above LSE on sieve is also known as \emph{KRR with inducing points}. We can notice that KRR with inducing points has two advantages compared with KRR. Firstly, it can achieve the optimal convergence rate for all choices of $\nu>0$; secondly, it has lower computational complexity -- the computational time complexity of KRR is $\CalO(n^3)$ while ours is $\CalO(n\CalS_n^2)=\CalO(n^{\frac{2s+3d}{2s+d}})$. For analysis of the computational time complexity of KRR with inducing points, please refer to \cite{davies2015effective,  ding2020generalization} for details.

\subsubsection{Risk Measures}

In \cite{wang2024smooth}, it has been shown that the convergence rate of $|\hat{f}_{(\lceil\frac{\tau}{n}\rceil)}-\text{VaR}_\tau(f(X))|$ cannot achieve the optimal rate $\Gamma^{-\frac{1}{2}}$  for any Mat\'ern-$\nu$ kernel with $\nu<\infty$. In this study, we discuss the case that the underlying $f$ resides in the RKHS $\CalF=\CalH_k$ induced by the Gaussian kernel $k(\BFx,\BFx')=\exp(-\omega\|\BFx-\BFx'\|^2)$, which can be treated as the limit $\nu\to\infty$. Similar to Section \ref{sec:KRR_nestedE},  we rewrite the KRR in the constrained optimization form \eqref{eq:constrained_KRR} with sieve $\CalF_{n,m}=\{h: \|h\|_{\CalH_k}^2\leqslant\CalR\}$ where $\|\cdot\|_{\CalH_k}$ is the RKHS norm induced by the Gaussian kernel. According to \cite{kuhn2011covering}, the $u$-entropy of $\CalF_{n,m}$ can be bounded as follows:
\[H(u,\CalF_{n,m},\|\cdot\|_n)=\CalO(\frac{(\log \frac{1}{u})^{d+1}}{(\log \log  \frac{1}{u})^d}).\]
Direct calculations show that the $\delta_{n,m}$ in Theorem \ref{thm:LSE_sieve} and $\Delta_n$ in Theorem \ref{thm:risk_convergence} are as follows:
\begin{align*}
    &\delta_{n,m}=\CalO\left((nm)^{-\frac{1}{2}}[\log nm]^{\frac{d+1}{2}}\right),\quad 
    \Delta_n= \CalO\left(n^{-\frac{1}{2}}[\log n]^{\frac{d+1}{2}}\right).
\end{align*}
Then, we can use \citep[Lemma E.4]{ding2024random} to get  embedding constant $\alpha=1$ for the Gaussian RKHS (omit a potential polylog term):
\[\|f-\hat{f}\|_{L_\infty}\leqslant {\CalO}(\|f-\hat{f}\|_{L_2}),\quad f-\hat{f}\in\CalF_{n,m}.\]
Therefore, if the parameters $\beta=\gamma=1$, we can have the almost optimal convergence rate by setting $m=\CalO(1)$:
\[|\hat{f}_{(\lceil\frac{\tau}{n}\rceil)}-\text{VaR}_\tau(f(X))|\leqslant\CalO_{\pr}(\Gamma^{-\frac{1}{2}}[\log\Gamma]^{\frac{d+1}{2}}).\]
To lower the computational time complexity, we choose $\CalS_n=[\log n]^{\frac{d}{2}}$ inducing points $\tilde{\BFX}=\{\tilde{\BFx}_j\}$ from $\{\BFx_i\}$ such that the filled distance of $\tilde{\BFx}_j$ satisfy
\begin{equation}
    \label{eq:inducing_filled_dist}
    h_{\Omega,\tilde{\BFX}}=\sup_{\BFx\in\Omega}\inf_{j\in[\CalS_n]}\|\tilde{\BFx}_j-\BFx\|=\CalO(\CalS_n^{-\frac{1}{d}})=\CalO([\log n]^{-\frac{1}{2}}).
\end{equation}
 Let $\CalF_{n,m}=\text{span}\{k(\tilde{\BFx}_j,\cdot)\}$ where $k$ is the Gaussian kernel. According to \cite[Theorem 11.22]{wendland2004scattered}, 
\[\inf_{h\in\CalF_{n,m}}\|h-f\|_n\leqslant\CalO(\exp[-1/h_{\Omega,\tilde{\BFX}}])=\CalO(n^{-\frac{1}{2}}).\]
Because $\CalF_{n,m}$ is a finite dimensional space, we can use \cite[Corollary 2.5]{geer2000empirical} to derive that
\[H(u,\CalF_n(\delta),\|\cdot\|_n)=[\log n]^{\frac{d}{2}}\log(\frac{4\delta+u}{u}).\]
Similarly, it can be shown that $\delta_{n,m}=\CalO(n^{-\frac{1}{2}}[\log n]^{\frac{d+1}{2}}$). By setting $n=\CalO(\Gamma)$ and $m=\CalO(1)$,  we have the same convergence rate $|\hat \theta_{n,m}-\theta| =\CalO_{\pr}(\Gamma^{-\frac{1}{2}}[\log \Gamma]^{\frac{d+1}{2}})$. Compared with the $\CalO(n^3)$ computational time complexity of KRR, the LSE on sieve only requires $\CalO(n\CalS_n^2)=\CalO(n[\log n]^{d+1})$ time to achieve the same convergence rate.
\begin{remark}
     Currently, no algorithm with theoretical guarantees exists for selecting an inducing point set that satisfies \eqref{eq:inducing_filled_dist}. However, as shown in \cite{zhang2021distance}, random designs exhibit properties similar to space-filling designs, effectively meeting the necessary conditions. Furthermore, \cite{brauchart2018random} proves that the expected filled distance of randomly selected points on a $d$-sphere satisfies a slightly weaker condition (with an extra log term), \( \mathit{E}[h_{\Omega, \tilde{\mathbf{X}}}] = \mathcal{O}((\log \mathcal{S}_n / \mathcal{S}_n)^{1/d}) \). Thus, both empirical and theoretical evidence suggest that randomly selecting inducing points generally suffices to meet these conditions in most regular domains.
\end{remark}

\subsection{Neural networks}\label{sec: ReLU}
We now apply our theories to obtain bounds for a specific type of neural networks called feed-forward multi-layer networks with ReLU activation. Here, we consider H\"older-$s$  underlying $f$ with smoothness parameter $s> d/2$, meaning that there exists some constant $M$ such that
\[\|f\|_{C^{s}(\Omega)}=\max_{|\BFv|<\lfloor s\rfloor}\|D^{\BFv}f\|_{L_\infty}+\max_{|\BFv|=\lfloor s\rfloor}\sup_{\BFx,\BFx'\in\Omega}\frac{|D^{\BFv}f(\BFx)-D^{\BFv}f(\BFx')|}{\|\BFx-\BFx'\|^{s-\lfloor s\rfloor}}\leqslant L,\]
where $\BFv=(v_1,\cdots,v_d)\in\NatInt^d$ consists of $d$ integers and 
\[D^{\BFv}f=\frac{\partial^{|\BFv|}f}{\partial x_1^{v_1}\cdots\partial x_d^{v_d}}.\]
Denote $A(x)=\max(x,0)$ as the ReLU activation function. An $H$-layer ReLU net $\Phi(H,W,S,B)$ is a parametrized function space consisting of functions of the form:
    $$\Phi(H,W,S,B)=\{f(\BFx)=[\BFW^{(H)}A(\cdot)+\BFb^{(H)}]\circ\cdots\circ[\BFW^{(1)}\BFx+\BFb^{(1)}]\}.$$
Here, $A(\cdot)$ is entry-wise positive part operation, $\BFW^{(1)}\in\Real^{W\times d}$,   $\BFW^{(h)}\in\Real^{W\times W}$ and $\BFb^{(h)}\in\Real^{W}$ for $h\in[H]$ and they satisfy
\begin{align*}
    &\sum_{h=1}^H\|\BFW^{(h)}\|_0+\|\BFb^{(h)}\|_0\leqslant S,\quad \max\{\max_{i,j,h}|\BFW^{(h)}_{i,j}|,\max_{i,h}|\BFb^{(h)}_i|\}\leqslant B,\\
\end{align*}
where $\|\BFM\|_{0}$ counts the number of non-zero elements of matrix $\BFM$. Intuitively, $H$ is the number of hidden layers, $W$ is the maximum size of weight matrices, $S$ is the sparsity, and $B$ is the maximum values allowed for the outputs of all activations. In summary, the target space $\CalF$ consists of all H\"older-$s$ functions
and the sieve $\CalF_{n,m}$ is a ReLU neural network:
\[\CalF=\{f: \|f\|_{\mathcal{C}^s(\Omega)}\leqslant L\},\quad \CalF_{n,m}=\Phi(H,W,S,B).\]
According to \cite{gouk2021regularisation}, $\Phi(H,W,S,B)$ is a subset of H\"older-1 space. Because it is also a parametric class, we can use the same reasoning in \cite[Example 19.7]{van1998asymptotic} to show that $\Phi(H,W,S,B)$ is Donsker. Moreover, we  apply  Gagliardo–Nirenberg interpolation inequality on H\"older-1 class, which leads to the following inequality for any function $f\in\CalF_{n,m}$:
\begin{equation}\label{eq:GNS_ineql}
    \|f\|_{L_\infty}\leqslant C_1\|f\|_{L_2}^{\frac{2}{d+2}}\|D^{\BFv}f\|_{L_\infty}^{\frac{d}{d+2}}\leqslant C_2\|f\|_{\mathcal{C}^1},
\end{equation}
where $C_1$ and $C_2$ are some universal constants. In summary,  $\Phi(H,W,S,B)$ satisfies Assumptions~\ref{assump:b_bounded} and \ref{assump:L_infinity_embedding}. For target function space $\CalF$, it is well known that H\"older-$s$ space can be embedded on Sobolev space $\mathscr{H}^s(\Omega)$ for compact $\Omega$. Therefore, $\CalF=\mathcal{C}^s(\Omega)$ also satisfies Assumptions~\ref{assump:b_bounded} and \ref{assump:L_infinity_embedding} for $s>d/2$.

By \citep[Lemma 3]{suzuki2018adaptivity}, the $u$-entropy of $\CalF_{n,m}$ can be bounded as follows:
\begin{equation*}
    H(u,\Phi(H,W,S,B),\|\cdot\|_n)\leqslant 2SH\log [(B\vee 1)(W+1)] +S\log\frac{H}{u}.
\end{equation*}
Therefore, it can be shown that the $\delta_{n,m}$ in Theorem \ref{thm:LSE_sieve} and the $\Delta_n$ in Theorem \ref{thm:risk_convergence} can be derived by
\begin{align*}
&\delta_{n,m}=\inf\{0<{\delta}<\frac{2^7\sigma}{\sqrt{m}}: \sqrt{nm}\delta^2\geqslant c\left( \sqrt{S}\int ^{\delta}_{\frac{\delta^2\sqrt{m}}{2^7\sigma}}\sqrt{\log\frac{H}{u} }\mathrm{d}u\vee \delta\right)\},\\
    &\Delta_n=\inf\left\{\delta: \frac{\delta^2}{b}\geqslant \inf_{\varepsilon\geqslant 0}\left\{4\varepsilon+\frac{C\sqrt{S}}{\sqrt{n}}\int_\varepsilon^{\delta} \sqrt{\log\frac{H}{u} }\mathrm{d}u\right\}\right\},
\end{align*}
respectively. From direct calculations, we can check that
\begin{equation}\label{eq:stat_err_relu}
    \delta_{n,m}=\CalO\left(\sqrt{\frac{S}{mn}\log[\frac{mnH^2}{S}]}\right),\quad \Delta_n=\CalO\left(\sqrt{\frac{b^2S}{n}}|\log[\sqrt{\frac{nH^2}{b^2S}}]|\right).
\end{equation}

For the approximation error of $\Phi(H,W,S,B)$, we have the following Proposition, which is adapted from \cite[Proposition 1]{suzuki2018adaptivity}

\begin{proposition}\label{prop:ReLU_approx}
  Let $C_0$ and $W_0$ be some constants that depend only on dimension $d$ and smoothness $s$ with $s\geqslant 1$. Let
\begin{align*}
    &H=3+2\lceil\log_2\frac{3^{d\vee\lfloor s+2\rfloor}}{\delta C_0}+5\rceil\lceil\log_2 d\vee\lfloor s+2\rfloor\rceil,\\
    &W=W_0 N,\quad 
    S=[(H-1)W_0^2+1]N,\quad B=N^{\frac{1}{d}},
\end{align*}  
where $N=\lceil |\delta\log \delta|^{-\frac{d}{s}}\rceil $. Then it holds that 
\[\sup_{f\in\CalF}\inf_{f^*\in\CalF_{n,m}}\|f-f^*\|_{n}\leqslant \sup_{f\in\CalF}\inf_{f^*\in\CalF_{n,m}}\|f-f^*\|_{L_\infty}=\CalO(\delta).\]
\end{proposition}
From Proposition \ref{prop:ReLU_approx}, it is clear that $H,W,S,B$ can all be treated as functions of the approximation error $\delta$. We now can apply Theorem \ref{thm:LSE_sieve} to get the convergence rate. We first let $m=\CalO(1)$ and 
$\delta=n^{-\frac{s}{2s+d}}\log n$ 
to get the associate $H,W,S,B$ accordingly. Then we substitute $H,W,S,B$ into \eqref{eq:stat_err_relu} to get 
\[\delta_{n,m}=\CalO(n^{-\frac{s}{2s+d}}\log n)=\Delta_n.\]
Therefore, the empirical error between $\hat{f}$ and $f$ is $\CalO_{\pr}(n^{-\frac{s}{2s+1}}\log n).$ However, the non-linearity term $\hat{\CalL}$ for neural net $\Phi(H,W,S,B)$ cannot be zero. This is because,  in contrast to the previous KRR examples, the set $\Phi(H,W,S,B)$ is not convex. Consequently, during the training process, there can be multiple resulting estimators ${\hat{f}_j}$ within $\mathcal{F}_{n,m}$ that closely approximate  $f$ on the dataset $\{\BFx_i\}$ but may deviate significantly from  $f$ outside of the dataset. However, the discussion of $\hat{\CalL}$ is beyond the scope of this study. We make the assumption that the training algorithm performs effectively, resulting in $\hat{\CalL}=\CalO(n^{-\frac{1}{2}})$. In this case, the nested expectation can achieve the optimal rate for $s>d/2$:
\[|\hat \theta_{n,m}-\theta|=|\frac{1}{n}
\sum_{i=1}^n\hat{f}\circ\eta(\BFx_i)-\mathit{E}[\eta\circ f(X)]| = O_{\pr}\left(n^{-\frac{1}{2}}\right)=\CalO_{\pr}(\Gamma^{-\frac{1}{2}}).\]

For the rate under risk measure, we use \eqref{eq:GNS_ineql} to get $\alpha=d/(d+2)$. Substitute $\alpha$, $\delta_{n,m}$, and $\Delta_n$ into Theorem \ref{thm:risk_convergence}, we have
\[|\hat{f}_{(\lceil\frac{\tau}{n}\rceil)}-\text{VaR}_\tau(f(X))| \leqslant\CalO(\Gamma^{-\frac{s}{2s+d}\frac{d}{d+2}\frac{\beta}{\gamma}}[\log \Gamma]^{-\frac{d}{d+2}\frac{\beta}{\gamma}}+n^{-\frac{1}{2\gamma}}).\]

\section{Numerical Experiments}\label{sec:numerical}
We now conduct numerical comparisons between the KRR estimator and the  LSE's on sieve, KRR with inducing points and ReLU neural network, as we introduced in Section \ref{sec:example}. In the first subsection, we conduct our numerical comparisons on synthetic data. Practical examples arising from portfolio risk management are discussed in the second subsection. All of the experiments are implemented in Matlab R2023a on a MacBook with Apple M2 Max Chip and 32GB of RAM.

\subsection{Synthetic Data}
 The first LSE utilizes inducing points $\{\tilde{\BFx}_i\}$ to construct sieve $\mathcal{F}_{n,m} = \text{span}\{k(\tilde{\BFx}_i,\cdot)\}$, while the second one employs sieve $\mathcal{F}_{n,m}$ as a three-layer ReLU neural network $\Phi$ with 256 activations in the first hidden layer and 128 activations in the second hidden layer. Furthermore, it's worth noting that throughout our experiments, we have set $m=1$ because generating duplicate samples at the same realization $\BFx_i$ does not lead to an improvement in the convergence rate, as we have demonstrated in Section \ref{sec:asymptotic}. The network is trained with the ADMM algorithm \citep{kingma2014adam} provided by the MATLAB deep learning toolbox \textit{adamupdate} which is the default option in MATLAB. All training was performed using MATLAB’s default options. 

 \begin{remark}
      We use the three-layer neural network because this architecture is based on its proven effectiveness in prior works, such as the original GAN paper \cite{goodfellow2014generative}, where a similar structure successfully generated complex structures. To prevent overfitting, the training process of the neural network in MATLAB adopts the early-stopping by dividing data into training set, validation set, and testing set. Specifically, the training process early stops when the performance on the validation set and the test set reaches a minimum. Early stop can be treated as a type of regularization, please refer to \cite{yao2007early,ding2024random} for details.
 \end{remark}

Let $f(\BFx) = \mathit{E}[Y|X=\BFx]$ be the conditional expectation in the inner level. 
We assume
\begin{equation} \label{eq:test-fun}
f(\BFx) = \frac{1}{N}\sum_{i=1}^N c_i k(\BFx, {\BFU}_i),\quad \BFx\in\Omega,
\end{equation}
where 
$N= 1\,000$, \textcolor{black}{$\Omega = [0,1]^d$}, and
$k$ is set  as the Laplace kernel $k_{\CalL}$ or the Gaussian kernel $k_{\CalG}$  $$k_{\CalL}(\BFx,\BFx')=e^{-\frac{\|\BFx-\BFx'\|}{d}},\quad k_{\CalG}(\BFx,\BFx')=e^{-\frac{\|\BFx-\BFx'\|^2}{d}}.$$
Both $c_i$ and $\BFU_i$ are randomly generated---with the former from \textcolor{black}{$\mathsf{Normal}(0, 1)$}  (i.e., the standard normal distribution) and the latter from  \textcolor{black}{$\mathsf{Uniform}[0,1]^d$}  (i.e., \textcolor{black}{the uniform distribution on \textcolor{black}{$[0,1]^d$}})---and then fixed. 
According to the definition of RKHSs, $f$ must be a function in the RKHS generated by the kernel $k$ because
\[\|f\|_{\CalH_k}^2=\frac{1}{N^2}\sum_{i,j=1}^Nc_ic_jk(\BFU_i,\BFU_j)<\infty.\]

All experiments can be separated into four classes as shown in the Table \ref{tab:experiment_classes}, so each experiment can be labeled by the vector $(\CalT,d)$.\\

\begin{table}[htbp]
    \centering
    {\small
    \begin{tabular}{@{\extracolsep{10pt}} c @{\extracolsep{10pt}} c @{\extracolsep{10pt}} c}
    \toprule
     & 10 dimension & 50 dimension \\
    \midrule
    $\CalT(Z)=Z^2$ & $k_{\CalL}$, $1\,000-5\,000$ data & $k_{\CalL}$, $4\,000-20\,000$ data \\ 
    $\CalT(Z)=\text{VaR}(Z)$ & $k_{\CalG}$, $1\,000-5\,000$ data & $k_{\CalG}$, $4\,000-20\,000$ data \\
    \bottomrule
    \end{tabular}
    }
    \caption{Classification of experiments by dimension and type.}
    \label{tab:experiment_classes}
\end{table}

We let outer-level scenario $X$ be uniformly distributed on $[0,1]^d$. Given a realization $\BFx_i$ of $X$, 
the inner-level samples are simulated via $y_{ij} = f(\BFx_i) + \varepsilon_{ij}$, where $\varepsilon_{ij}$'s are i.i.d. \textcolor{black}{$\mathsf{Normal}(0, 1)$} random variables. Once an estimator $\hat{\theta}=\CalT\circ \hat{f} (X)$ of $\theta=\CalT\circ f(X)$ is constructed, we sample 10\,000 points for $X$ to estimate the absolute value $|\hat{\theta}-\theta|$ for comparison.

For experiments in dimension $d=10$, data sizes of $n=1\,000, 2\,000, \cdots, 5\,000$ are used. In the case of dimension $d=50$, data sizes of $n=4\,000,\cdots,20\,000$ are used. Each experiment is repeated 1\,000 times and we  report the averaged absolute error, its standard deviation 
(STD) when $\CalT(Z)=Z^2$, and one-tenth of its STD when $\CalT(Z)=\text{VaR}(Z)$  as our result. We chose to display one-tenth of the STD associated with the ReLU neural network due to its relatively large size. Displaying its original STD would obscure other results, making the chart difficult to interpret.

For experiments with functional $\CalT(Z)=Z^2$,  the underlying function $f$ in \eqref{eq:test-fun} is a linear combination of Laplace kernel $k_\CalL$. In this case, the RKHS generated by $k_\CalL$ is  the $(\frac{1}{2}+\frac{d}{2})$-order Sobolev space with $d=10$ or $50$. This aligns with the case where the smoothness parameter $\nu = \frac{1}{2} < \frac{d}{2}$, and consequently, KRR with inducing points and ReLU neural network both exhibit a faster convergence rate (given that the training algorithm of neural net performs well enough), as discussed in Sections \ref{sec:KRR_nestedE} and \ref{sec: ReLU}. The KRR with inducing point has inducing point set randomly selected without replacement from data $\{\BFx_i\}$ with size $\CalS_n=n^{\frac{1}{2}}$. The ReLU neural net is introduced at the beginning of this section. As demonstrated in Figure \ref{fig:result} (a) and (d), KRR with inducing points consistently outperforms its competitors. The ReLU network performs similarly to KRR when the dimension is $d=10$ and surpasses KRR when $d=50$. According to our results in Section \ref{sec:example}, both KRR with inducing points and the ReLU neural network exhibit a faster convergence rate than KRR in this case. For KRR with inducing points, our experiments successfully validate this assertion. However, for the ReLU network, we believe that its suboptimal performance is due to its high $u$-entropy when dealing with small datasets $n=1\,000,\cdots,5\,000$. When the data size is sufficiently large, i.e. $n=4\,000,\cdots,20\,000$, the ReLU network outperforms KRR and shows similar performance to KRR with inducing points.

For experiments with functional $\CalT(Z)=\text{VaR}(Z)$,  the underlying function $f$ in \eqref{eq:test-fun} is a linear combination of Gaussian kernels $k_\mathcal{G}$. In this case, the underlying function $f$ exhibits smoothness. As discussed in Section \ref{sec:example}, both KRR and KRR with inducing points can achieve the optimal convergence rate of $\mathcal{O}(\Gamma^{-\frac{1}{2}})$. In our experiments, KRR with inducing points selects the inducing point set randomly without replacement from the data ${\BFx_i}$, with a size of $\mathcal{S}_n = |\log n|^3$. From Figure \ref{fig:result} (b) and (e), it is evident that KRR with inducing points consistently outperforms KRR. We attribute this to the fact that KRR with inducing points requires significantly lower computational resources due to $\mathcal{S}_n = |\log n|^3$, resulting in much lower numerical error compared to KRR. As for the ReLU neural network, it is not a smooth estimator, and thus, it can never attain the optimal convergence rate, regardless of the smoothness of the underlying function $f$ as discussed in Section \ref{sec: ReLU}. Consequently, it consistently exhibits the poorest performance among all estimators. Table \ref{tab:log_log} presents the average slope of the logarithmic error versus data size for data sizes between  \( n=4\,000 \) and  \( n=5\,000 \). This slope approximates the exponent of the convergence rate. Our results indicate that KRR with inducing points achieves a convergence rate of approximately \( \mathcal{O}(n^{-1/2}) \) across all scenarios, aligning with our theoretical statement.
\begin{figure}[!t]
\centering
\subfloat[ ]{\includegraphics[width=0.45\textwidth, height=0.23\textheight]{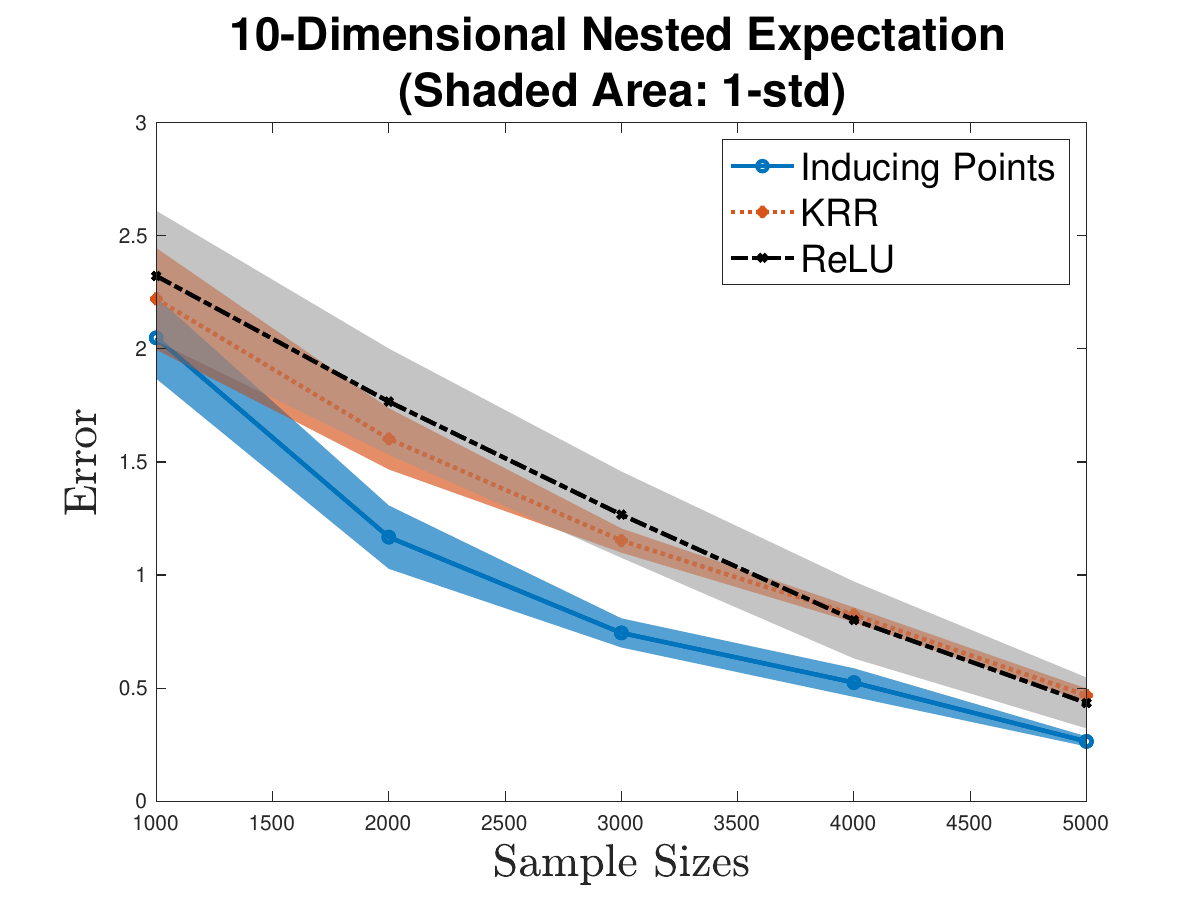}}
\subfloat[]{\includegraphics[width=0.45\textwidth, height=0.23\textheight]{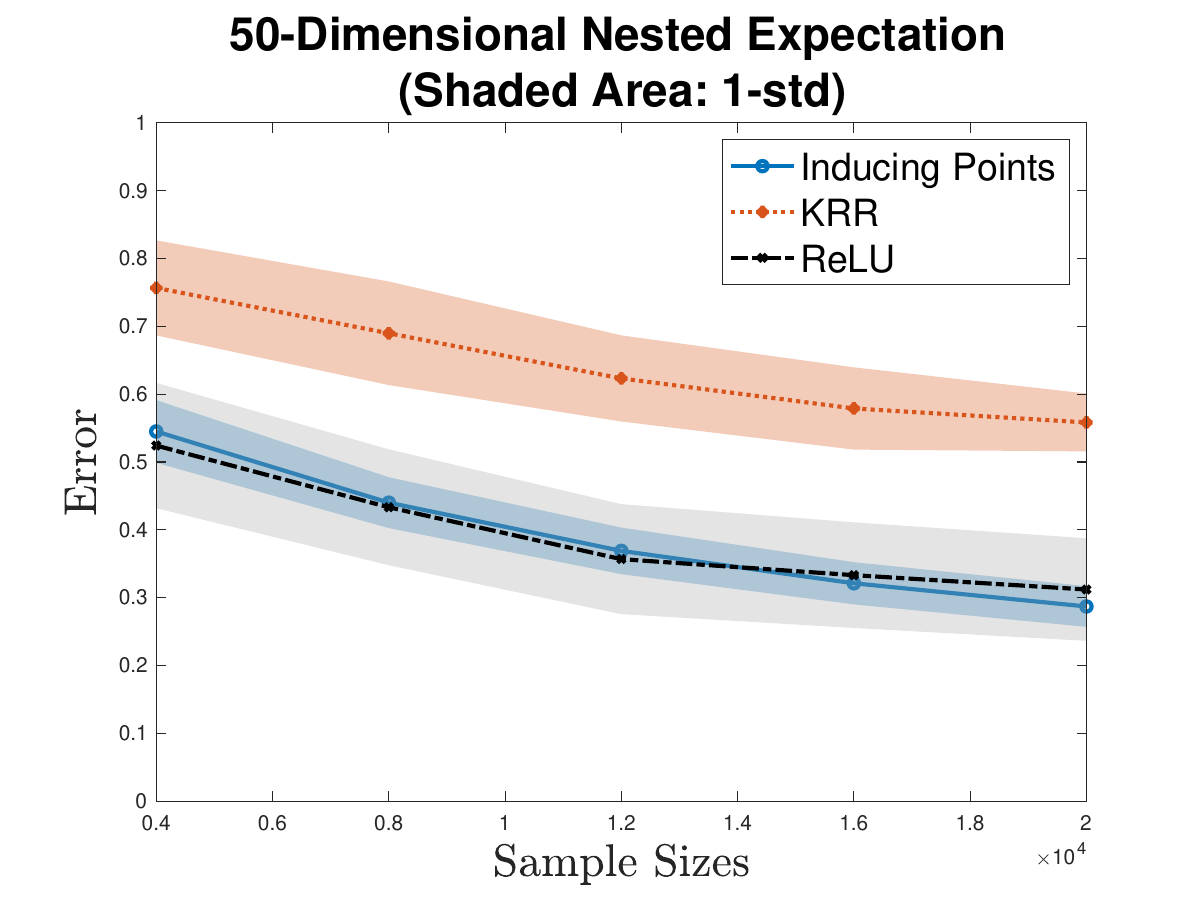}}\\
\subfloat[]{\includegraphics[width=0.45\textwidth, height=0.23\textheight]{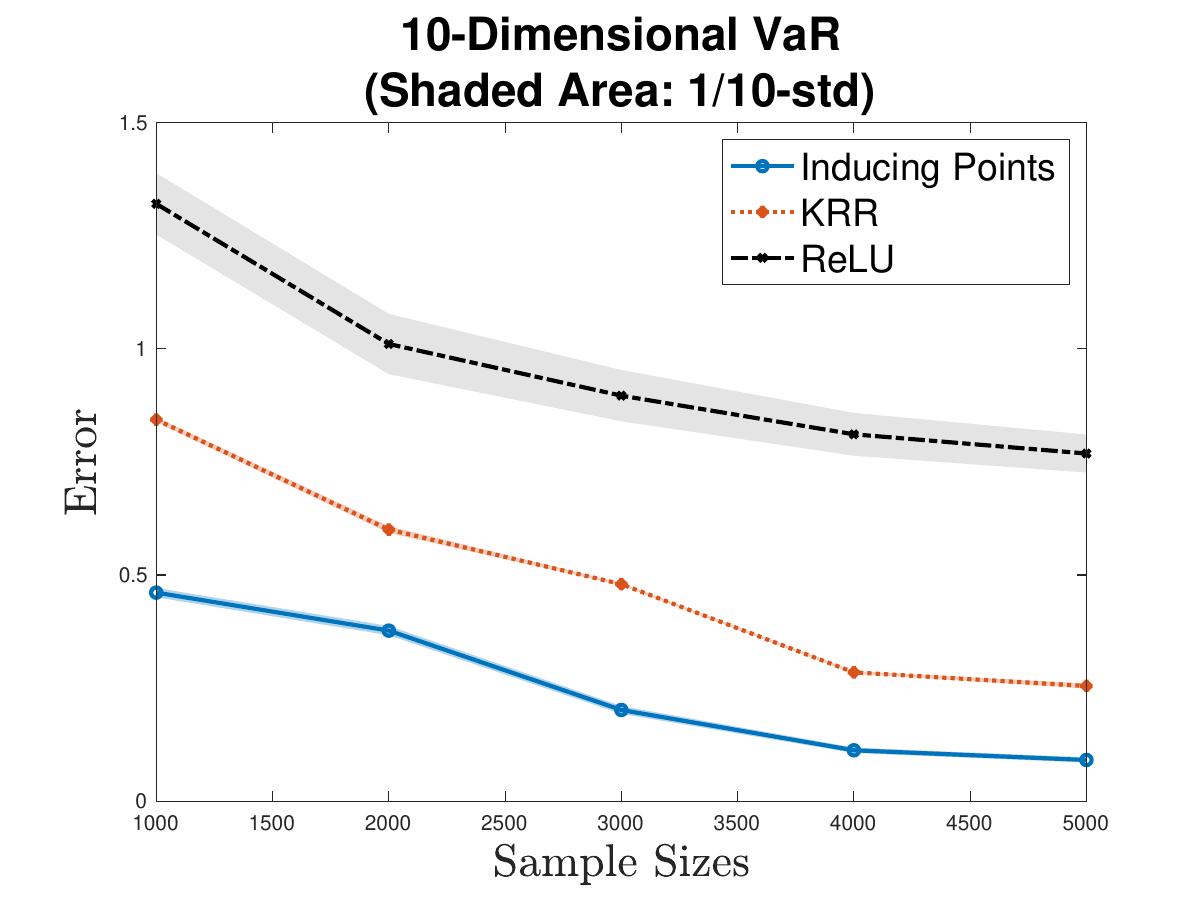}}
\subfloat[]{\includegraphics[width=0.45\textwidth, height=0.23\textheight]{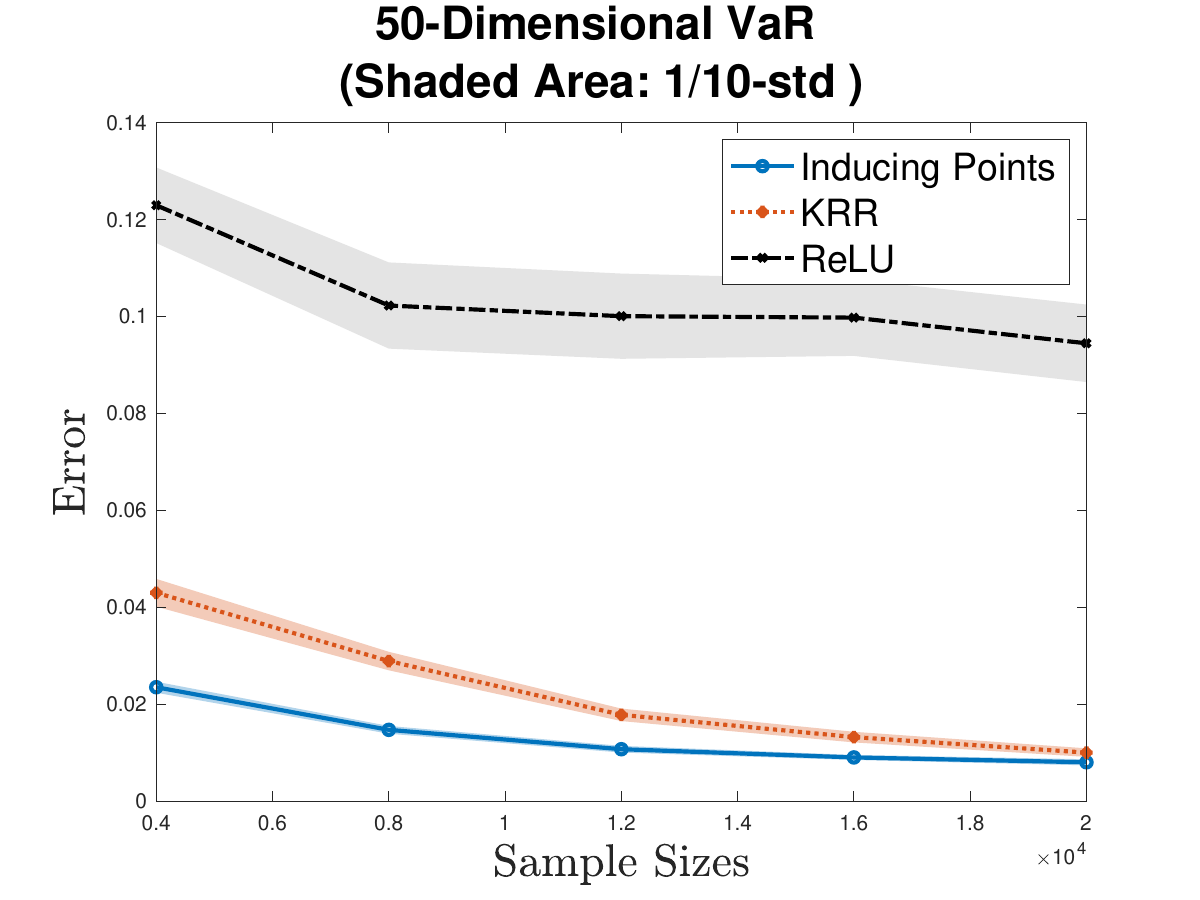}}\\
\subfloat[]{\includegraphics[width=0.45\textwidth, height=0.23\textheight]{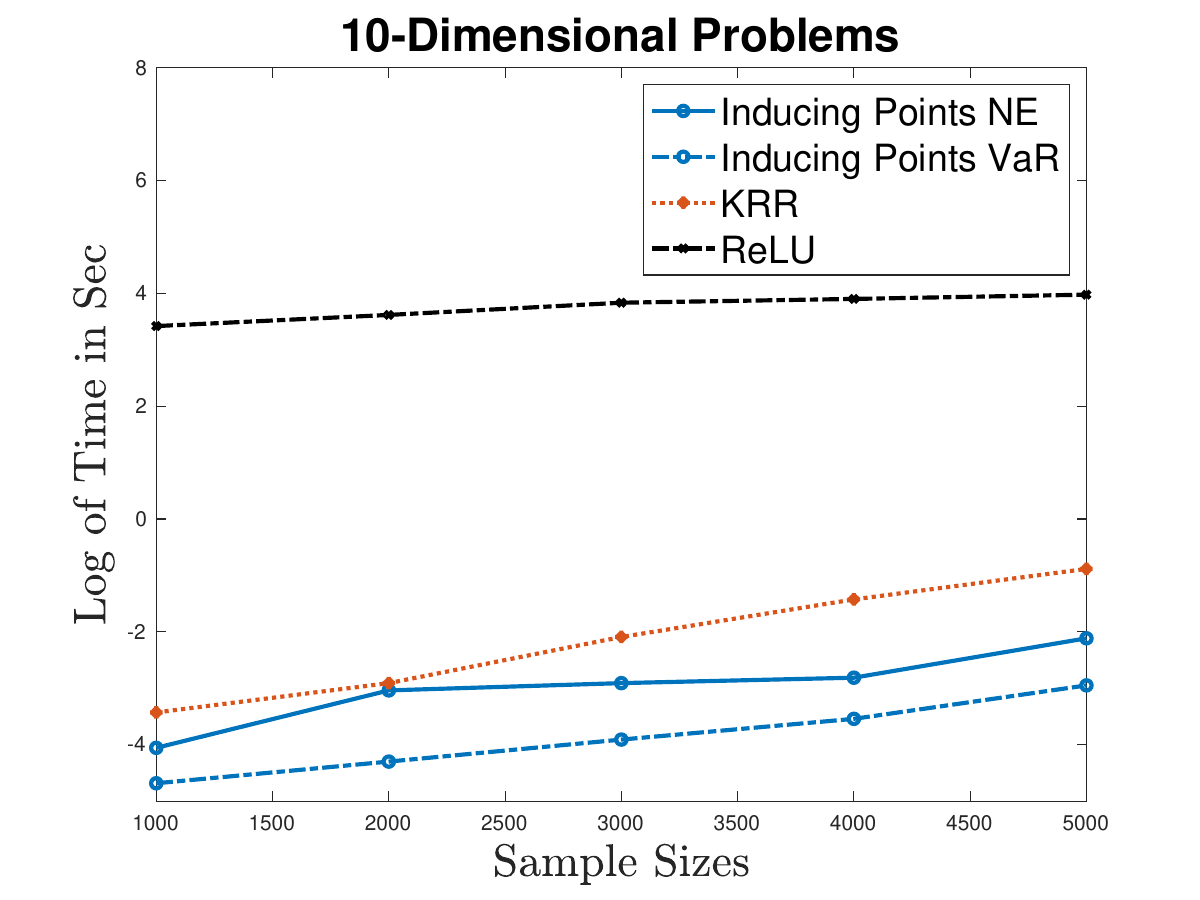}}
\subfloat[]{\includegraphics[width=0.45\textwidth, height=0.23\textheight]{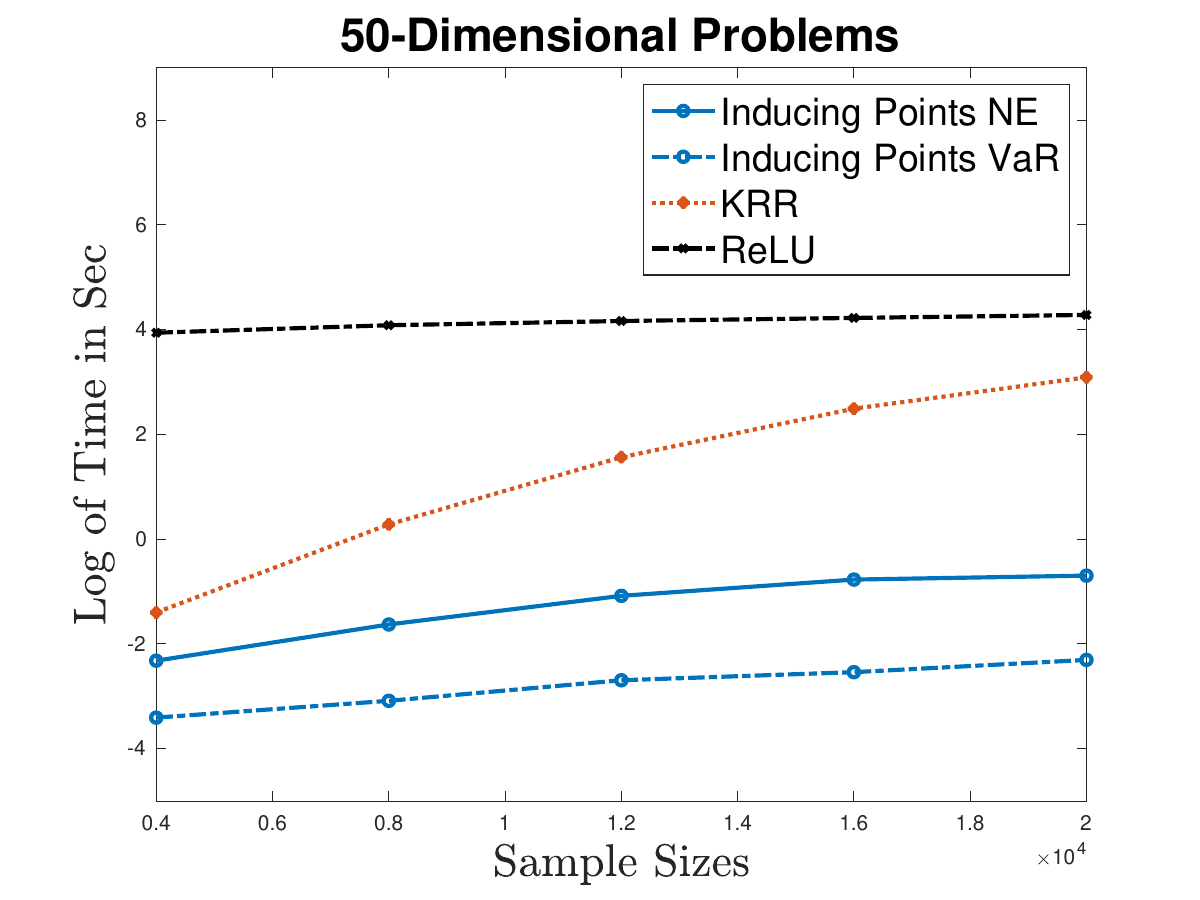}}\\

\caption{Upper row: MSE error of experiments associated to $\CalT(Z)=Z^2$; middle:  MSE errors of experiments  associated to  $\CalT(Z)=\text{VaR}(Z)$; bottom row: Computational time of experiments. The lines represent the averaged errors and the shaded area represents the STD for nested expectation and one-tenth of the STD for VaR.\label{fig:result}}
\end{figure}

\begin{table}[htbp]
    \centering
    {\small
    \begin{tabular}{@{\extracolsep{10pt}} l @{\extracolsep{10pt}} l @{\extracolsep{10pt}} r r @{\extracolsep{10pt}} r r @{\extracolsep{10pt}} r r}
    \toprule
    & \multirow{2.5}{*}{$\frac{\log(|\hat{\theta} -\theta|)-\log(|\hat{\theta} -\theta|)}{\log(5000)-\log(4000)}$} & \multicolumn{3}{c}{Nested Expectation} & \multicolumn{3}{c}{VaR} \\
    \cmidrule{3-8}
    & & KRR & KIP & ReLU & KRR & KIP & ReLU \\
    \cmidrule{1-2} \cmidrule{3-5} \cmidrule{6-8}
     & 10-dimension & -.572 & -.581 & -.534  & -.500 & -.615 & -.241 \\
    & 50-dimension & -.161 & -.532 & -.297 & -.527 & -.492  & -.121 \\
    \bottomrule
    \end{tabular}
    }
    \caption{The average slope of the logarithmic error versus logarithmic data size for data sizes between  \( n=4\,000 \) and  \( n=5\,000 \) is calculated to illustrate the convergence rates of different methods.   \label{tab:log_log} }
\end{table}

We have also plotted the time required for training each estimator, as shown in Figure \ref{fig:result} (c) and (f). The training times for the KRR and ReLU neural networks are similar for both nested expectation and VaR. Therefore, we report the average training times for KRR and the neural network over all experiments with dimension $d=10$. Regarding KRR with inducing points, in the nested expectation case, Laplace kernel is used so there are $n^{\frac{1}{2}}$ inducing points, while in the VaR case, Gaussian kernel is used so there are $|\log n|^3$ inducing points. Therefore, we separately plot these different scenarios. The running times are plotted on a logarithmic scale. It's evident that Gaussian kernel KRR with inducing points has the shortest running time among all cases. This is because it only requires $|\log n|^3$ inducing points, making it a highly computationally efficient estimator.

\subsection{Portfolio Risk Management}

In this subsection, we explore simulations of a portfolio composed of options, each dependent on the performance of $q$ underlying assets, with a shared expiration date at $T$. The portfolio manager aims to evaluate the risk at a future point in time, $T_0 < T$, using five distinct metrics: the expected quadratic loss $\mathit{E}[Z^2]$, the expected excess loss $\mathit{E}[(Z-z_0)^+]$, and the VaR of $Z$ at a specific risk level $\tau$. Here, $Z$ represents the portfolio loss at time $T_0$, and $z_0$ denotes a predetermined threshold. These three metrics are formalized through the transformation $\CalT(\cdot) = \mathit{E}[\eta(\cdot)]$, where $\eta(z) = z^2$, $\eta(z) = (z - z_0)^+$, and $\eta(z) = \inf\{z \in \Real : \Pr(Z \leqslant z) \geqslant \tau\}$, respectively.

Let $\BFS(t) = [S_1(t), \cdots, S_q(t)]$ represent the vector of asset prices at time $t$. Within the nested simulation framework, we first simulate $\BFS(t)$ up to time $T_0$, with these sample paths forming the scenarios at the outer level. Subsequently, for the inner-level simulation, we project $\BFS(t)$ from $T_0$ to $T$ to approximate the values of the options within the portfolio. An important nuance is the distinction in probability measures applied at the outer and inner simulation levels; the outer-level simulation employs the real-world measure, while the inner-level adopts a risk-neutral measure. Assume $\BFS(t)$ adheres to a $q$-dimensional geometric Brownian motion (GBM) model:
\[\frac{dS_i(t)}{S_i(t)} = \mu_i dt + \sum_{j=1}^q \sigma_{ij} \mathrm{d}B_j(t), \quad i = 1, \cdots, q,\]
where $B_i(t)$ are independent standard one-dimensional Brownian motions, the matrix $[\sigma_{ij}]_{i,j}$ is lower-triangular (i.e., $\sigma_{ij} = 0$ for all $i > j$), and $\mu_i$ varies based on the probability measure in use. For simplicity, we assume a uniform return $\mu$ for each underlying asset under the real-world measure (i.e., $\mu_i = \mu$ for all $i$). Conversely, under the risk-neutral measure, $\mu_i$ should align with the risk-free interest rate $r$ (i.e., $\mu_i = r$ for all $i$).

The portfolio comprises six options for each underlying asset, totaling 6q options. Included within these are three geometric Asian call options with discrete monitoring and three up-and-out barrier call options with continuous monitoring, which can be knocked out anytime from $0$ to $T$. Both option types feature path-dependent payoffs. For an asset with price $S(t)$, the payoff for a geometric Asian call option, monitored at times $0 = t_0 < t_1 < \cdots < t_M = T$, is defined as $((\prod_{k=1}^M S(t_k))^{1/M} - K)^+$, where $K$ is the strike price. The payoff for an up-and-out barrier option with barrier $H$ is $(S(T) - K)^+ \rmI\{\max_{0 \leqslant t \leqslant T} S(t) \leqslant H\}$. The strike prices for the three Asian options on each asset are $K_1$, $K_2$, and $K_3$, which similarly apply to the three barrier options; these barrier options also share the same barrier level $H$. The risk horizon $T_0 = t_{M_0}$ is defined for some $M_0 = 1, \cdots, M$. According to derivative pricing theory \citep[Chapter 1]{Glasserman03}, the value of each option at $T_0$ is its expected discounted payoff, leading to the portfolio's value at $T_0$ being

\begin{align*}
    V_{T_0}(\BFX) = & \mathit{E} \left[e^{-r(T-T_0)}\sum_{i=1}^q\sum_{l=1}^3\left(((\prod_{k=1}^M S_i(t_k))^{1/M}-K_l)^+ \right. \right. \\
    & \left. \left. + (S_i(T)-K_l)^+\rmI\{\max_{0\leqslant t \leqslant T} S(t) \leqslant H\}\right) \bigg| \BFX\right],
\end{align*}

where $\BFX \in \Real^{3q}$ is a vector representing the risk factors, defined as

\begin{align*}
    \BFX = & \left[S_1(T_0), \cdots, S_q(T_0), (\prod_{k=1}^M S_1(t_k))^{1/M}, \cdots, (\prod_{k=1}^M S_q(t_k))^{1/M}, \right. \\
     & \quad \left. \max_{0 \leqslant t \leqslant T_0} S_1(t), \cdots, \max_{0 \leqslant t \leqslant T_0} S_q(t)\right].
\end{align*}
Thus, the portfolio's loss at time $T_0$ is $Z = V_0 - V_{T_0}(\BFX)$, where $V_0$ is the portfolio's initial value at time $0$, calculated similarly to $V_{T_0}$ but with $T_0 = 0$. This loss can also be represented as $Z = \mathit{E}[Y|\BFX]$, where $Y = V_0 - V(\BFX)$.

To assess the performance of a nested simulation method, we need to compute the true value of $\Theta=\CalT(Z)$, which is done as follows. Under the assumption that $\BFS(t)$ follows a GBM, $V_{T_0}(\BFX)$ can be calculated in closed form (see e.g., Chapter 4 in \cite{haug2007complete}). We generate $10^8$ i.i.d. copies of $\BFX$ and calculate the corresponding conditional expectation  $V_0-V_{T_0}(\BFX)$, which are i.i.d. copies of $Z$ and can be used to accurately estimate $\theta$.

We assess various methods using the relative root mean squared error (RRMSE), defined as the RMSE ratio over an accurate estimate of $\theta$, along with the sample standard deviation (STD). For each problem instance, RRMSE and sample variance estimates are derived from 1\,000 macro-replications. The RRMSE reflects the accuracy of an estimator while the sample STD reflects its stability. 

The additional parameters are defined as follows: the maturity date for each option in the portfolio is $T = 1$, with a risk horizon at $T_0 = 3/50$. Within the GBM framework, the initial price for each asset is set at $S_i(0) = 100$ for $i = 1, \cdots, q$. The return for each asset under the real-world measure is $\mu = 8\%$, the risk-free rate is $r = 5\%$, and the volatility $\sigma_{i,j}$ is determined randomly (detailed generation process in Appendix \ref{append:sigma_generation}) and fixed for $i \leqslant j$, with $\sigma_{ij} = 0$ for $i > j$. The values of $q$ considered are 10, 20, 50, and 100, resulting in a dimensionality of $d = 3q$ for the conditioning variable $\BFX$.

For the portfolio options, there are three distinct strike prices: $K_1 = 90$, $K_2 = 100$, and $K_3 = 110$. The monitoring points for each Asian option are set at $\{t_k = kT/M : k = 1, \cdots, M\}$, where $M = 50$, and the barrier level for each barrier option is $H = 150$.

Similar to the previous subsection, we compare the following three methods for nested simulation with a budget of $10^5$ and the settings are exactly the same as the previous subsection:
\begin{enumerate}
    \item Kernel Ridge Regression. We try different values of $m$ and report the best performance.
    \item Kernel Ridge Regression with inducing points. The  inducing points are randomly selected without replacement from data $\{\BFx_i\}$ with size  $\CalS_n=n^{\frac{1}{2}}$.
    \item A three-layer ReLU neural network $\Phi$ with 256 activations in the first hidden layer and 128 activations in the second hidden layer.
\end{enumerate}

Table \ref{tab:Risk_Mange_result} presents the RRMSE and sample STD results, where $m$ is chosen to be 10 for all competing methods. Here we do not consider the standard nested simulation method, which ignores the correlation of the inner simulation. This is because it has been shown in \cite{wang2024smooth} that KRR and regression-based method all outperform the standard method.  The inducing point method consistently outperforms the other two methods in terms of both RRMSE and sample STD. The computational times of all models are similar to those shown in  Figure \ref{fig:result} (f) for sample size $n=10^4$. The KRR method is significantly lower than all its competitors for its computation involves numerical inversion of n×n matrices, which is an intensive computational task for large $n$.
\begin{table}[htbp]
    \centering
    {\small
    \begin{tabular}{@{\extracolsep{10pt}} l @{\extracolsep{10pt}} l @{\extracolsep{10pt}} r r @{\extracolsep{10pt}} r r @{\extracolsep{10pt}} r r}
    \toprule
    & \multirow{2.5}{*}{$\mathcal{T}$} & \multicolumn{3}{c}{RRMSE} & \multicolumn{3}{c}{Sample STD} \\
    \cmidrule{3-8}
    & & KRR & KIP & ReLU & KRR & KIP & ReLU \\
    \cmidrule{1-2} \cmidrule{3-5} \cmidrule{6-8}
    \multirow{3}{*}{$q = 10$ } & Quadratic & 1.32 & .626 & .834  & .743 & .515 & 1.14 \\
    & Hockey-stick & 1.56 & .732 & 1.64 & .081 & .052  & .121 \\
    & VaR & .993 & .851 & 2.88 & .872 & .616 & 4.32 \\
    \cmidrule{1-2} \cmidrule{3-5} \cmidrule{6-8}
    \multirow{3}{*}{$q = 20$} & Quadratic & 1.27 & .763 & .818 & .895 & .714 & 1.01 \\
    & Hockey-stick & 1.78 & 1.26 & 1.72 & .172 & .093 & .271 \\
    & VaR & 1.17 & .932 & 2.95 & .892  & .623 & 4.48 \\
    \cmidrule{1-2} \cmidrule{3-5} \cmidrule{6-8}
    \multirow{3}{*}{$q = 50$} & Quadratic & 1.26 & .815 & 1.12 & 1.04 & .839 & 2.78 \\
    & Hockey-stick & 1.55 & 1.15 & 1.44 & .221 & .133 & .345 \\
    & VaR & 1.13 & .992 & 3.57 & 1.01 & .653 & 4.77 \\
    \cmidrule{1-2} \cmidrule{3-5} \cmidrule{6-8}
    \multirow{3}{*}{$q = 100$} & Quadratic & 1.31 & 1.02 & 1.14 & 1.65 & 1.12 & 2.09 \\
    & Hockey-stick & 1.88 & 1.16 & 1.92 & .539 & .274 & .626 \\
    & VaR & 1.23 & 1.09 & 3.90 & 1.16 & .978 & 5.39 \\
    \bottomrule
    \end{tabular}
    }
    \caption{The dimensionality of the conditioning variable is $d=3q$. For hockey-stick, the threshold is  $z_0=0.02 V_0$. For VaR, the risk level is $\tau=99\%$. KIP stands for KRR with inducing points. \label{tab:Risk_Mange_result} }
\end{table}

The KRR with inducing points not only demonstrates the lowest computational cost but also exhibits low RRMSE and low sample standard deviation (STD), as illustrated in Table \ref{tab:Risk_Mange_result}. For instance, in scenarios where $q=100$, the total running time ratio of the KRR with inducing points is significantly lower than that of the standard KRR. Furthermore, as detailed in Table \ref{tab:Risk_Mange_result}, with identical values of $\Gamma$ and $q$, both the RRMSE and sample STD of the KRR with inducing points are at least $20\%$ superior to those of the standard KRR when $\eta$ corresponds to a quadratic function or the VaR. This indicates that despite being a low-rank approximation, the KRR with inducing points can achieve at least the same convergence rate as the full KRR. In summary, KRR with inducing points is capable of yielding considerable time savings while maintaining a high level of estimation accuracy, particularly in high-dimensional settings and when dealing with a large number of simulation samples.

\section{Conclusion}\label{sec:conclusion}

In this paper, we generalize the KRR for nested simulation to LSE on sieve. We demonstrate that, for various scenarios of nested simulation, more accurate LSEs on sieve exist, surpassing KRR and the traditional Monte Carlo method in terms of performance. We relax the conditions required for KRR to achieve the square root convergence rate, which is the standard rate for conventional Monte Carlo simulations. Our theoretical framework for convergence rate analysis is more versatile and can accommodate different forms of nested simulation.

Our work can be extended in several ways. Firstly, we have shown that the convergence rate of a neural network also depends on the training algorithm. However, how to select a training algorithm to achieve optimal convergence rate is unknown. Furthermore, there may be opportunities for enhancements in our upper bounds on the convergence rates for the VaR functional. This is primarily because the convergence rate under the VaR metric is closely related to the structure, or more precisely, the topology of the class $\CalF$ and the sieve $\CalF_{n,m}$. It is evident that the concept ``entropy'' in our study may not be the ideal tool for capturing the topology information, so achieving better results than those presented in \cite{wang2024smooth} for the convergence rate under VaR may be challenging.



\vspace{1em}

{\flushleft\textbf{Acknowledgment:} The authors thank the associate editor and anonymous reviewers for their helpful comments and suggestions, which have led to a significantly improved paper. In particular, the authors are deeply grateful to one of the reviewers for their meticulous and thorough review, as well as for providing detailed feedback that greatly enhanced the quality and clarity of this work.}
\vspace{1em}
{\flushleft\textbf{Conflict of interest:} Authors have no conflict of funding and completing interests to declare.}
\vspace{1em}
{\flushleft \textbf{Author contributions:} Ruoxue Liu:  numerical experiments, writing; Liang Ding: analysis, numerical experiments; Wenjia Wang: methodology, funding acquisition; Lu Zou: methodology, analysis, writing.}

\bibliography{JORSC-bibliography}


\begin{thebibliography}{62}
\ifx \bisbn   \undefined \def \bisbn  #1{ISBN #1}\fi
\ifx \binits  \undefined \def \binits#1{#1}\fi
\ifx \bauthor  \undefined \def \bauthor#1{#1}\fi
\ifx \batitle  \undefined \def \batitle#1{#1}\fi
\ifx \bjtitle  \undefined \def \bjtitle#1{#1}\fi
\ifx \bvolume  \undefined \def \bvolume#1{\textbf{#1}}\fi
\ifx \byear  \undefined \def \byear#1{#1}\fi
\ifx \bissue  \undefined \def \bissue#1{#1}\fi
\ifx \bfpage  \undefined \def \bfpage#1{#1}\fi
\ifx \blpage  \undefined \def \blpage #1{#1}\fi
\ifx \burl  \undefined \def \burl#1{\textsf{#1}}\fi
\ifx \doiurl  \undefined \def \doiurl#1{\url{https://doi.org/#1}}\fi
\ifx \betal  \undefined \def \betal{\textit{et al.}}\fi
\ifx \binstitute  \undefined \def \binstitute#1{#1}\fi
\ifx \binstitutionaled  \undefined \def \binstitutionaled#1{#1}\fi
\ifx \bctitle  \undefined \def \bctitle#1{#1}\fi
\ifx \beditor  \undefined \def \beditor#1{#1}\fi
\ifx \bpublisher  \undefined \def \bpublisher#1{#1}\fi
\ifx \bbtitle  \undefined \def \bbtitle#1{#1}\fi
\ifx \bedition  \undefined \def \bedition#1{#1}\fi
\ifx \bseriesno  \undefined \def \bseriesno#1{#1}\fi
\ifx \blocation  \undefined \def \blocation#1{#1}\fi
\ifx \bsertitle  \undefined \def \bsertitle#1{#1}\fi
\ifx \bsnm \undefined \def \bsnm#1{#1}\fi
\ifx \bsuffix \undefined \def \bsuffix#1{#1}\fi
\ifx \bparticle \undefined \def \bparticle#1{#1}\fi
\ifx \barticle \undefined \def \barticle#1{#1}\fi
\bibcommenthead
\ifx \bconfdate \undefined \def \bconfdate #1{#1}\fi
\ifx \botherref \undefined \def \botherref #1{#1}\fi
\ifx \url \undefined \def \url#1{\textsf{#1}}\fi
\ifx \bchapter \undefined \def \bchapter#1{#1}\fi
\ifx \bbook \undefined \def \bbook#1{#1}\fi
\ifx \bcomment \undefined \def \bcomment#1{#1}\fi
\ifx \oauthor \undefined \def \oauthor#1{#1}\fi
\ifx \citeauthoryear \undefined \def \citeauthoryear#1{#1}\fi
\ifx \endbibitem  \undefined \def \endbibitem {}\fi
\ifx \bconflocation  \undefined \def \bconflocation#1{#1}\fi
\ifx \arxivurl  \undefined \def \arxivurl#1{\textsf{#1}}\fi
\csname PreBibitemsHook\endcsname

\bibitem[\protect\citeauthoryear{Adams and Fournier}{2003}]{adams2003sobolev}
\begin{bbook}
\bauthor{\bsnm{Adams}, \binits{R.A.}},
\bauthor{\bsnm{Fournier}, \binits{J.J.}}:
\bbtitle{Sobolev Spaces},
\bedition{2nd} edn.
\bpublisher{Academic Press},
\blocation{Cambridge, MA}
(\byear{2003})
\end{bbook}
\endbibitem

\bibitem[\protect\citeauthoryear{Asmussen and Glynn}{2007}]{AsmussenGlynn07}
\begin{bbook}
\bauthor{\bsnm{Asmussen}, \binits{S.}},
\bauthor{\bsnm{Glynn}, \binits{P.W.}}:
\bbtitle{Stochastic Simulation: Algorithm and Analysis}.
\bpublisher{Springer},
\blocation{NY}
(\byear{2007})
\end{bbook}
\endbibitem

\bibitem[\protect\citeauthoryear{Broadie et~al.}{2015a}]{BroadieDuMoallemi15}
\begin{barticle}
\bauthor{\bsnm{Broadie}, \binits{M.}},
\bauthor{\bsnm{Du}, \binits{Y.}},
\bauthor{\bsnm{Moallemi}, \binits{C.C.}}:
\batitle{Risk estimation via regression}.
\bjtitle{Oper. Res.}
\bvolume{63}(\bissue{5}),
\bfpage{1077}--\blpage{1097}
(\byear{2015})
\end{barticle}
\endbibitem

\bibitem[\protect\citeauthoryear{Broadie et~al.}{2015b}]{broadie2015risk}
\begin{barticle}
\bauthor{\bsnm{Broadie}, \binits{M.}},
\bauthor{\bsnm{Du}, \binits{Y.}},
\bauthor{\bsnm{Moallemi}, \binits{C.C.}}:
\batitle{Risk estimation via regression}.
\bjtitle{Operations Research}
\bvolume{63}(\bissue{5}),
\bfpage{1077}--\blpage{1097}
(\byear{2015})
\end{barticle}
\endbibitem

\bibitem[\protect\citeauthoryear{Blanchard and M{\"u}cke}{2018}]{blanchard2018optimal}
\begin{barticle}
\bauthor{\bsnm{Blanchard}, \binits{G.}},
\bauthor{\bsnm{M{\"u}cke}, \binits{N.}}:
\batitle{Optimal rates for regularization of statistical inverse learning problems}.
\bjtitle{Foundations of Computational Mathematics}
\bvolume{18}(\bissue{4}),
\bfpage{971}--\blpage{1013}
(\byear{2018})
\end{barticle}
\endbibitem

\bibitem[\protect\citeauthoryear{Barton et~al.}{2014}]{barton2014quantifying}
\begin{barticle}
\bauthor{\bsnm{Barton}, \binits{R.R.}},
\bauthor{\bsnm{Nelson}, \binits{B.L.}},
\bauthor{\bsnm{Xie}, \binits{W.}}:
\batitle{Quantifying input uncertainty via simulation confidence intervals}.
\bjtitle{INFORMS journal on computing}
\bvolume{26}(\bissue{1}),
\bfpage{74}--\blpage{87}
(\byear{2014})
\end{barticle}
\endbibitem

\bibitem[\protect\citeauthoryear{Brauchart et~al.}{2018}]{brauchart2018random}
\begin{barticle}
\bauthor{\bsnm{Brauchart}, \binits{J.S.}},
\bauthor{\bsnm{Reznikov}, \binits{A.B.}},
\bauthor{\bsnm{Saff}, \binits{E.B.}},
\bauthor{\bsnm{Sloan}, \binits{I.H.}},
\bauthor{\bsnm{Wang}, \binits{Y.G.}},
\bauthor{\bsnm{Womersley}, \binits{R.S.}}:
\batitle{Random point sets on the sphere—hole radii, covering, and separation}.
\bjtitle{Experimental Mathematics}
\bvolume{27}(\bissue{1}),
\bfpage{62}--\blpage{81}
(\byear{2018})
\end{barticle}
\endbibitem

\bibitem[\protect\citeauthoryear{Burt et~al.}{2019}]{burt2019rates}
\begin{bchapter}
\bauthor{\bsnm{Burt}, \binits{D.}},
\bauthor{\bsnm{Rasmussen}, \binits{C.E.}},
\bauthor{\bsnm{Van Der~Wilk}, \binits{M.}}:
\bctitle{Rates of convergence for sparse variational gaussian process regression}.
In: \bbtitle{International Conference on Machine Learning},
pp. \bfpage{862}--\blpage{871}
(\byear{2019}).
\bcomment{PMLR}
\end{bchapter}
\endbibitem

\bibitem[\protect\citeauthoryear{Birman and Solomjak}{1967}]{Birman_1967}
\begin{barticle}
\bauthor{\bsnm{Birman}, \binits{M.S.}},
\bauthor{\bsnm{Solomjak}, \binits{M.Z.}}:
\batitle{Piecewise-polynomial approximations of functions of the classes}.
\bjtitle{Mathematics of the USSR-Sbornik}
\bvolume{2}(\bissue{3}),
\bfpage{295}
(\byear{1967})
\doiurl{10.1070/SM1967v002n03ABEH002343}
\end{barticle}
\endbibitem

\bibitem[\protect\citeauthoryear{Ch{\'e}rief-Abdellatif}{2020}]{cherief2020convergence}
\begin{bchapter}
\bauthor{\bsnm{Ch{\'e}rief-Abdellatif}, \binits{B.-E.}}:
\bctitle{Convergence rates of variational inference in sparse deep learning}.
In: \bbtitle{International Conference on Machine Learning},
pp. \bfpage{1831}--\blpage{1842}
(\byear{2020}).
\bcomment{PMLR}
\end{bchapter}
\endbibitem

\bibitem[\protect\citeauthoryear{Cheng et~al.}{2022}]{cheng2022constructing}
\begin{barticle}
\bauthor{\bsnm{Cheng}, \binits{H.-F.}},
\bauthor{\bsnm{Liu}, \binits{X.}},
\bauthor{\bsnm{Zhang}, \binits{K.}}:
\batitle{Constructing confidence intervals for nested simulation}.
\bjtitle{Naval Research Logistics (NRL)}
\bvolume{69}(\bissue{8}),
\bfpage{1138}--\blpage{1149}
(\byear{2022})
\end{barticle}
\endbibitem

\bibitem[\protect\citeauthoryear{Cheng and Zhang}{2021}]{cheng2021non}
\begin{barticle}
\bauthor{\bsnm{Cheng}, \binits{H.-F.}},
\bauthor{\bsnm{Zhang}, \binits{K.}}:
\batitle{Non-nested estimators for the central moments of a conditional expectation and their convergence properties}.
\bjtitle{Operations Research Letters}
\bvolume{49}(\bissue{5}),
\bfpage{625}--\blpage{632}
(\byear{2021})
\end{barticle}
\endbibitem

\bibitem[\protect\citeauthoryear{Davies}{2015}]{davies2015effective}
\begin{botherref}
\oauthor{\bsnm{Davies}, \binits{A.J.}}:
Effective implementation of gaussian process regression for machine learning.
PhD thesis,
University of Cambridge
(2015)
\end{botherref}
\endbibitem

\bibitem[\protect\citeauthoryear{Dicker et~al.}{2017}]{dicker2017kernel}
\begin{barticle}
\bauthor{\bsnm{Dicker}, \binits{L.H.}},
\bauthor{\bsnm{Foster}, \binits{D.P.}},
\bauthor{\bsnm{Hsu}, \binits{D.}}:
\batitle{Kernel ridge vs. principal component regression: Minimax bounds and the qualification of regularization operators}.
\bjtitle{Electronic Journal of Statistics}
\bvolume{11}(\bissue{1}),
\bfpage{1022}--\blpage{1047}
(\byear{2017})
\end{barticle}
\endbibitem

\bibitem[\protect\citeauthoryear{Ding et~al.}{2024}]{ding2024random}
\begin{barticle}
\bauthor{\bsnm{Ding}, \binits{L.}},
\bauthor{\bsnm{Hu}, \binits{T.}},
\bauthor{\bsnm{Jiang}, \binits{J.}},
\bauthor{\bsnm{Li}, \binits{D.}},
\bauthor{\bsnm{Wang}, \binits{W.}},
\bauthor{\bsnm{Yao}, \binits{Y.}}:
\batitle{Random smoothing regularization in kernel gradient descent learning}.
\bjtitle{Journal of Machine Learning Research}
\bvolume{25}(\bissue{284}),
\bfpage{1}--\blpage{88}
(\byear{2024})
\end{barticle}
\endbibitem

\bibitem[\protect\citeauthoryear{Donoho}{1995}]{donoho1995noising}
\begin{barticle}
\bauthor{\bsnm{Donoho}, \binits{D.L.}}:
\batitle{De-noising by soft-thresholding}.
\bjtitle{IEEE transactions on information theory}
\bvolume{41}(\bissue{3}),
\bfpage{613}--\blpage{627}
(\byear{1995})
\end{barticle}
\endbibitem

\bibitem[\protect\citeauthoryear{Ding et~al.}{2020}]{ding2020generalization}
\begin{bchapter}
\bauthor{\bsnm{Ding}, \binits{L.}},
\bauthor{\bsnm{Tuo}, \binits{R.}},
\bauthor{\bsnm{Shahrampour}, \binits{S.}}:
\bctitle{Generalization guarantees for sparse kernel approximation with entropic optimal features}.
In: \bbtitle{International Conference on Machine Learning},
pp. \bfpage{2545}--\blpage{2555}
(\byear{2020}).
\bcomment{PMLR}
\end{bchapter}
\endbibitem

\bibitem[\protect\citeauthoryear{Evans}{2022}]{evans2022partial}
\begin{bbook}
\bauthor{\bsnm{Evans}, \binits{L.C.}}:
\bbtitle{Partial Differential Equations}
vol. \bseriesno{19}.
\bpublisher{American Mathematical Society},
\blocation{Providence}
(\byear{2022})
\end{bbook}
\endbibitem

\bibitem[\protect\citeauthoryear{Fu and Hu}{1997}]{FuHu97}
\begin{bbook}
\bauthor{\bsnm{Fu}, \binits{M.C.}},
\bauthor{\bsnm{Hu}, \binits{J.-Q.}}:
\bbtitle{Conditional Monte Carlo: Gradient Estimation and Optimization Applications}.
\bpublisher{Springer},
\blocation{NY}
(\byear{1997})
\end{bbook}
\endbibitem

\bibitem[\protect\citeauthoryear{Fu et~al.}{2009}]{FuHongHu09}
\begin{barticle}
\bauthor{\bsnm{Fu}, \binits{M.C.}},
\bauthor{\bsnm{Hong}, \binits{L.J.}},
\bauthor{\bsnm{Hu}, \binits{J.-Q.}}:
\batitle{Conditional {Monte Carlo} estimation of quantile sensitivities}.
\bjtitle{Manag. Sci.}
\bvolume{55}(\bissue{12}),
\bfpage{2019}--\blpage{2027}
(\byear{2009})
\end{barticle}
\endbibitem

\bibitem[\protect\citeauthoryear{Feng and Song}{2020}]{feng2020optimal}
\begin{botherref}
\oauthor{\bsnm{Feng}, \binits{M.B.}},
\oauthor{\bsnm{Song}, \binits{E.}}:
Optimal nested simulation experiment design via likelihood ratio method.
arXiv preprint arXiv:2008.13087
(2020)
\end{botherref}
\endbibitem

\bibitem[\protect\citeauthoryear{Gouk et~al.}{2021}]{gouk2021regularisation}
\begin{barticle}
\bauthor{\bsnm{Gouk}, \binits{H.}},
\bauthor{\bsnm{Frank}, \binits{E.}},
\bauthor{\bsnm{Pfahringer}, \binits{B.}},
\bauthor{\bsnm{Cree}, \binits{M.J.}}:
\batitle{Regularisation of neural networks by enforcing lipschitz continuity}.
\bjtitle{Machine Learning}
\bvolume{110},
\bfpage{393}--\blpage{416}
(\byear{2021})
\end{barticle}
\endbibitem

\bibitem[\protect\citeauthoryear{Gordy and Juneja}{2010}]{GordyJuneja10}
\begin{barticle}
\bauthor{\bsnm{Gordy}, \binits{M.B.}},
\bauthor{\bsnm{Juneja}, \binits{S.}}:
\batitle{Nested simulation in portfolio risk measurement}.
\bjtitle{Manag. Sci.}
\bvolume{56}(\bissue{10}),
\bfpage{1833}--\blpage{1848}
(\byear{2010})
\end{barticle}
\endbibitem

\bibitem[\protect\citeauthoryear{Glasserman}{2003}]{Glasserman03}
\begin{bbook}
\bauthor{\bsnm{Glasserman}, \binits{P.}}:
\bbtitle{Monte Carlo Methods in Financial Engineering},
\bedition{1st} edn.
\bpublisher{Springer},
\blocation{New York}
(\byear{2003})
\end{bbook}
\endbibitem

\bibitem[\protect\citeauthoryear{Goodfellow et~al.}{2014}]{goodfellow2014generative}
\begin{botherref}
\oauthor{\bsnm{Goodfellow}, \binits{I.}},
\oauthor{\bsnm{Pouget-Abadie}, \binits{J.}},
\oauthor{\bsnm{Mirza}, \binits{M.}},
\oauthor{\bsnm{Xu}, \binits{B.}},
\oauthor{\bsnm{Warde-Farley}, \binits{D.}},
\oauthor{\bsnm{Ozair}, \binits{S.}},
\oauthor{\bsnm{Courville}, \binits{A.}},
\oauthor{\bsnm{Bengio}, \binits{Y.}}:
Generative adversarial nets.
Advances in neural information processing systems
\textbf{27}
(2014)
\end{botherref}
\endbibitem

\bibitem[\protect\citeauthoryear{Guntuboyina and Sen}{2018}]{guntuboyina2018nonparametric}
\begin{barticle}
\bauthor{\bsnm{Guntuboyina}, \binits{A.}},
\bauthor{\bsnm{Sen}, \binits{B.}}:
\batitle{Nonparametric shape-restricted regression}.
\bjtitle{Statistical Science}
\bvolume{33}(\bissue{4}),
\bfpage{568}--\blpage{594}
(\byear{2018})
\end{barticle}
\endbibitem

\bibitem[\protect\citeauthoryear{Hastie}{2020}]{Hastie20}
\begin{barticle}
\bauthor{\bsnm{Hastie}, \binits{T.}}:
\batitle{Ridge regularization: An essential concept in data science}.
\bjtitle{Technometrics}
\bvolume{62}(\bissue{4}),
\bfpage{426}--\blpage{433}
(\byear{2020})
\end{barticle}
\endbibitem

\bibitem[\protect\citeauthoryear{Haug}{2007}]{haug2007complete}
\begin{bbook}
\bauthor{\bsnm{Haug}, \binits{E.G.}}:
\bbtitle{The Complete Guide to Option Pricing Formulas}.
\bpublisher{McGraw-Hill},
\blocation{New York}
(\byear{2007})
\end{bbook}
\endbibitem

\bibitem[\protect\citeauthoryear{Hong and Juneja}{2009}]{hong2009estimating}
\begin{bchapter}
\bauthor{\bsnm{Hong}, \binits{L.J.}},
\bauthor{\bsnm{Juneja}, \binits{S.}}:
\bctitle{Estimating the mean of a non-linear function of conditional expectation}.
In: \bbtitle{Proceedings of the 2009 Winter Simulation Conference (WSC)},
pp. \bfpage{1223}--\blpage{1236}
(\byear{2009}).
\bcomment{IEEE}
\end{bchapter}
\endbibitem

\bibitem[\protect\citeauthoryear{Hong et~al.}{2017}]{hong2017kernel}
\begin{barticle}
\bauthor{\bsnm{Hong}, \binits{L.J.}},
\bauthor{\bsnm{Juneja}, \binits{S.}},
\bauthor{\bsnm{Liu}, \binits{G.}}:
\batitle{Kernel smoothing for nested estimation with application to portfolio risk measurement}.
\bjtitle{Operations Research}
\bvolume{65}(\bissue{3}),
\bfpage{657}--\blpage{673}
(\byear{2017})
\end{barticle}
\endbibitem

\bibitem[\protect\citeauthoryear{Jacot et~al.}{2018}]{jacot2018neural}
\begin{botherref}
\oauthor{\bsnm{Jacot}, \binits{A.}},
\oauthor{\bsnm{Gabriel}, \binits{F.}},
\oauthor{\bsnm{Hongler}, \binits{C.}}:
Neural tangent kernel: Convergence and generalization in neural networks.
Advances in neural information processing systems
\textbf{31}
(2018)
\end{botherref}
\endbibitem

\bibitem[\protect\citeauthoryear{Kingma and Ba}{2014}]{kingma2014adam}
\begin{botherref}
\oauthor{\bsnm{Kingma}, \binits{D.P.}},
\oauthor{\bsnm{Ba}, \binits{J.}}:
Adam: A method for stochastic optimization.
arXiv preprint arXiv:1412.6980
(2014)
\end{botherref}
\endbibitem

\bibitem[\protect\citeauthoryear{Kanagawa et~al.}{2018}]{KanHenSejSri18}
\begin{botherref}
\oauthor{\bsnm{Kanagawa}, \binits{M.}},
\oauthor{\bsnm{Hennig}, \binits{P.}},
\oauthor{\bsnm{Sejdinovic}, \binits{D.}},
\oauthor{\bsnm{Sriperumbudur}, \binits{B.K.}}:
Gaussian Processes and Kernel Methods: A Review on Connections and Equivalences.
Preprint available at \url{https://arxiv.org/abs/1807.02582}
(2018)
\end{botherref}
\endbibitem

\bibitem[\protect\citeauthoryear{Kohler et~al.}{2009}]{kohler2009optimal}
\begin{barticle}
\bauthor{\bsnm{Kohler}, \binits{M.}},
\bauthor{\bsnm{Krzy{\.z}ak}, \binits{A.}},
\bauthor{\bsnm{Walk}, \binits{H.}}:
\batitle{Optimal global rates of convergence for nonparametric regression with unbounded data}.
\bjtitle{Journal of Statistical Planning and Inference}
\bvolume{139}(\bissue{4}),
\bfpage{1286}--\blpage{1296}
(\byear{2009})
\end{barticle}
\endbibitem

\bibitem[\protect\citeauthoryear{Koltchinskii}{2011}]{koltchinskii2011oracle}
\begin{bbook}
\bauthor{\bsnm{Koltchinskii}, \binits{V.}}:
\bbtitle{Oracle Inequalities in Empirical Risk Minimization and Sparse Recovery Problems: {\'E}cole D’{\'E}t{\'e} de Probabilit{\'e}s de Saint-Flour XXXVIII-2008}
vol. \bseriesno{2033}.
\bpublisher{Springer},
\blocation{New York}
(\byear{2011})
\end{bbook}
\endbibitem

\bibitem[\protect\citeauthoryear{Kuchibhotla and Patra}{2022}]{kuchibhotla2022least}
\begin{barticle}
\bauthor{\bsnm{Kuchibhotla}, \binits{A.K.}},
\bauthor{\bsnm{Patra}, \binits{R.K.}}:
\batitle{On least squares estimation under heteroscedastic and heavy-tailed errors}.
\bjtitle{The Annals of Statistics}
\bvolume{50}(\bissue{1}),
\bfpage{277}--\blpage{302}
(\byear{2022})
\end{barticle}
\endbibitem

\bibitem[\protect\citeauthoryear{Kur and Rakhlin}{2021}]{kur2021minimal}
\begin{bchapter}
\bauthor{\bsnm{Kur}, \binits{G.}},
\bauthor{\bsnm{Rakhlin}, \binits{A.}}:
\bctitle{On the minimal error of empirical risk minimization}.
In: \bbtitle{Conference on Learning Theory},
pp. \bfpage{2849}--\blpage{2852}
(\byear{2021}).
\bcomment{PMLR}
\end{bchapter}
\endbibitem

\bibitem[\protect\citeauthoryear{K{\"u}hn}{2011}]{kuhn2011covering}
\begin{barticle}
\bauthor{\bsnm{K{\"u}hn}, \binits{T.}}:
\batitle{Covering numbers of gaussian reproducing kernel hilbert spaces}.
\bjtitle{Journal of Complexity}
\bvolume{27}(\bissue{5}),
\bfpage{489}--\blpage{499}
(\byear{2011})
\end{barticle}
\endbibitem

\bibitem[\protect\citeauthoryear{Lan et~al.}{2010}]{lan2010confidence}
\begin{barticle}
\bauthor{\bsnm{Lan}, \binits{H.}},
\bauthor{\bsnm{Nelson}, \binits{B.L.}},
\bauthor{\bsnm{Staum}, \binits{J.}}:
\batitle{A confidence interval procedure for expected shortfall risk measurement via two-level simulation}.
\bjtitle{Operations Research}
\bvolume{58}(\bissue{5}),
\bfpage{1481}--\blpage{1490}
(\byear{2010})
\end{barticle}
\endbibitem

\bibitem[\protect\citeauthoryear{Liu and Staum}{2010}]{liu2010stochastic}
\begin{barticle}
\bauthor{\bsnm{Liu}, \binits{M.}},
\bauthor{\bsnm{Staum}, \binits{J.}}:
\batitle{Stochastic kriging for efficient nested simulation of expected shortfall}.
\bjtitle{Journal of Risk}
\bvolume{12}(\bissue{3}),
\bfpage{3}
(\byear{2010})
\end{barticle}
\endbibitem

\bibitem[\protect\citeauthoryear{Lin and Yang}{2020a}]{lin2020fast}
\begin{barticle}
\bauthor{\bsnm{Lin}, \binits{X.S.}},
\bauthor{\bsnm{Yang}, \binits{S.}}:
\batitle{Fast and efficient nested simulation for large variable annuity portfolios: A surrogate modeling approach}.
\bjtitle{Insurance: Mathematics and Economics}
\bvolume{91},
\bfpage{85}--\blpage{103}
(\byear{2020})
\end{barticle}
\endbibitem

\bibitem[\protect\citeauthoryear{Lin and Yang}{2020b}]{LinYang20}
\begin{barticle}
\bauthor{\bsnm{Lin}, \binits{X.S.}},
\bauthor{\bsnm{Yang}, \binits{S.}}:
\batitle{Fast and efficient nested simulation for large variable annuity portfolios: A surrogate modeling approach}.
\bjtitle{Insurance: Mathematics and Economics}
\bvolume{91},
\bfpage{85}--\blpage{103}
(\byear{2020})
\end{barticle}
\endbibitem

\bibitem[\protect\citeauthoryear{Liu et~al.}{2024}]{liu2024kernel}
\begin{barticle}
\bauthor{\bsnm{Liu}, \binits{X.}},
\bauthor{\bsnm{Yan}, \binits{X.}},
\bauthor{\bsnm{Zhang}, \binits{K.}}:
\batitle{Kernel quantile estimators for nested simulation with application to portfolio value-at-risk measurement}.
\bjtitle{European Journal of Operational Research}
\bvolume{312}(\bissue{3}),
\bfpage{1168}--\blpage{1177}
(\byear{2024})
\end{barticle}
\endbibitem

\bibitem[\protect\citeauthoryear{Liu and Zhou}{2019}]{liu2019online}
\begin{botherref}
\oauthor{\bsnm{Liu}, \binits{T.}},
\oauthor{\bsnm{Zhou}, \binits{E.}}:
Online quantification of input model uncertainty by two-layer importance sampling.
arXiv preprint arXiv:1912.11172
(2019)
\end{botherref}
\endbibitem

\bibitem[\protect\citeauthoryear{Liang et~al.}{2024}]{liang2024fast}
\begin{botherref}
\oauthor{\bsnm{Liang}, \binits{G.}},
\oauthor{\bsnm{Zhang}, \binits{K.}},
\oauthor{\bsnm{Luo}, \binits{J.}}:
A fast method for nested estimation.
INFORMS Journal on Computing
(2024)
\end{botherref}
\endbibitem

\bibitem[\protect\citeauthoryear{Rakhlin et~al.}{2017}]{rakhlin2017empirical}
\begin{botherref}
\oauthor{\bsnm{Rakhlin}, \binits{A.}},
\oauthor{\bsnm{Spidharan}, \binits{K.}},
\oauthor{\bsnm{Tsybakov}, \binits{A.B.}}:
Empirical entropy, minimax regret and minimax risk.
Bernoulli,
789--824
(2017)
\end{botherref}
\endbibitem

\bibitem[\protect\citeauthoryear{Suzuki}{2019}]{suzuki2018adaptivity}
\begin{bchapter}
\bauthor{\bsnm{Suzuki}, \binits{T.}}:
\bctitle{Adaptivity of deep relu network for learning in besov and mixed smooth besov spaces: optimal rate and curse of dimensionality}.
In: \bbtitle{Proceedings of the Nternational Conference on Learning Representations}
(\byear{2019})
\end{bchapter}
\endbibitem

\bibitem[\protect\citeauthoryear{van~de Geer}{2000}]{geer2000empirical}
\begin{bbook}
\bauthor{\bsnm{Geer}, \binits{S.}}:
\bbtitle{Empirical Processes in M-Estimation}.
\bpublisher{Cambridge University Press},
\blocation{Cambridge}
(\byear{2000})
\end{bbook}
\endbibitem

\bibitem[\protect\citeauthoryear{van~de Geer}{2014}]{van2014uniform}
\begin{barticle}
\bauthor{\bsnm{Geer}, \binits{S.}}:
\batitle{{On the uniform convergence of empirical norms and inner products, with application to causal inference}}.
\bjtitle{Electronic Journal of Statistics}
\bvolume{8}(\bissue{1}),
\bfpage{543}--\blpage{574}
(\byear{2014})
\end{barticle}
\endbibitem

\bibitem[\protect\citeauthoryear{van~der Vaart}{1998}]{van1998asymptotic}
\begin{bbook}
\bauthor{\bsnm{Vaart}, \binits{A.W.}}:
\bbtitle{Asymptotic Statistics}.
\bpublisher{Cambridge University Press},
\blocation{Cambridge}
(\byear{1998})
\end{bbook}
\endbibitem

\bibitem[\protect\citeauthoryear{Van Der~Vaart and Wellner}{1997}]{van1997weak}
\begin{bbook}
\bauthor{\bsnm{Van Der~Vaart}, \binits{A.W.}},
\bauthor{\bsnm{Wellner}, \binits{J.A.}}:
\bbtitle{Weak Convergence and Empirical Processes: with Applications to Statistics}.
\bpublisher{Springer},
\blocation{New York}
(\byear{1997})
\end{bbook}
\endbibitem

\bibitem[\protect\citeauthoryear{Wainwright}{2019}]{Wainwright19}
\begin{bbook}
\bauthor{\bsnm{Wainwright}, \binits{M.J.}}:
\bbtitle{High-Dimensional Statistics: A Non-Asymptotic Viewpoint}.
\bpublisher{Cambridge University Press},
\blocation{Cambridge}
(\byear{2019})
\end{bbook}
\endbibitem

\bibitem[\protect\citeauthoryear{Wendland}{2004}]{wendland2004scattered}
\begin{bbook}
\bauthor{\bsnm{Wendland}, \binits{H.}}:
\bbtitle{Scattered Data Approximation}.
\bpublisher{Cambridge University Press},
\blocation{Cambridge}
(\byear{2004})
\end{bbook}
\endbibitem

\bibitem[\protect\citeauthoryear{Wu and Schaback}{1993}]{wu1993local}
\begin{barticle}
\bauthor{\bsnm{Wu}, \binits{Z.-m.}},
\bauthor{\bsnm{Schaback}, \binits{R.}}:
\batitle{Local error estimates for radial basis function interpolation of scattered data}.
\bjtitle{IMA journal of Numerical Analysis}
\bvolume{13}(\bissue{1}),
\bfpage{13}--\blpage{27}
(\byear{1993})
\end{barticle}
\endbibitem

\bibitem[\protect\citeauthoryear{Wang et~al.}{2024}]{wang2024smooth}
\begin{botherref}
\oauthor{\bsnm{Wang}, \binits{W.}},
\oauthor{\bsnm{Wang}, \binits{Y.}},
\oauthor{\bsnm{Zhang}, \binits{X.}}:
Smooth nested simulation: Bridging cubic and square root convergence rates in high dimensions.
Management Science
(2024)
\end{botherref}
\endbibitem

\bibitem[\protect\citeauthoryear{Xie et~al.}{2014}]{xie2014bayesian}
\begin{barticle}
\bauthor{\bsnm{Xie}, \binits{W.}},
\bauthor{\bsnm{Nelson}, \binits{B.L.}},
\bauthor{\bsnm{Barton}, \binits{R.R.}}:
\batitle{A bayesian framework for quantifying uncertainty in stochastic simulation}.
\bjtitle{Operations Research}
\bvolume{62}(\bissue{6}),
\bfpage{1439}--\blpage{1452}
(\byear{2014})
\end{barticle}
\endbibitem

\bibitem[\protect\citeauthoryear{Yao et~al.}{2007}]{yao2007early}
\begin{barticle}
\bauthor{\bsnm{Yao}, \binits{Y.}},
\bauthor{\bsnm{Rosasco}, \binits{L.}},
\bauthor{\bsnm{Caponnetto}, \binits{A.}}:
\batitle{On early stopping in gradient descent learning}.
\bjtitle{Constructive Approximation}
\bvolume{26}(\bissue{2}),
\bfpage{289}--\blpage{315}
(\byear{2007})
\end{barticle}
\endbibitem

\bibitem[\protect\citeauthoryear{Zhang et~al.}{2021}]{zhang2021distance}
\begin{barticle}
\bauthor{\bsnm{Zhang}, \binits{B.}},
\bauthor{\bsnm{Cole}, \binits{D.A.}},
\bauthor{\bsnm{Gramacy}, \binits{R.B.}}:
\batitle{Distance-distributed design for gaussian process surrogates}.
\bjtitle{Technometrics}
\bvolume{63}(\bissue{1}),
\bfpage{40}--\blpage{52}
(\byear{2021})
\end{barticle}
\endbibitem

\bibitem[\protect\citeauthoryear{Zhang et~al.}{2022}]{zhang2022sample}
\begin{botherref}
\oauthor{\bsnm{Zhang}, \binits{K.}},
\oauthor{\bsnm{Feng}, \binits{B.M.}},
\oauthor{\bsnm{Liu}, \binits{G.}},
\oauthor{\bsnm{Wang}, \binits{S.}}:
Sample recycling for nested simulation with application in portfolio risk measurement.
arXiv preprint arXiv:2203.15929
(2022)
\end{botherref}
\endbibitem

\bibitem[\protect\citeauthoryear{Zhang et~al.}{2022a}]{zhang2022bootstrap}
\begin{barticle}
\bauthor{\bsnm{Zhang}, \binits{K.}},
\bauthor{\bsnm{Liu}, \binits{G.}},
\bauthor{\bsnm{Wang}, \binits{S.}}:
\batitle{Bootstrap-based budget allocation for nested simulation}.
\bjtitle{Operations Research}
\bvolume{70}(\bissue{2}),
\bfpage{1128}--\blpage{1142}
(\byear{2022})
\end{barticle}
\endbibitem

\bibitem[\protect\citeauthoryear{Zhang et~al.}{2022b}]{ZhangLiuWang22}
\begin{barticle}
\bauthor{\bsnm{Zhang}, \binits{K.}},
\bauthor{\bsnm{Liu}, \binits{G.}},
\bauthor{\bsnm{Wang}, \binits{S.}}:
\batitle{Technical note---{Bootstrap}-based budget allocation for nested simulation}.
\bjtitle{Oper. Res.}
\bvolume{70}(\bissue{2}),
\bfpage{1128}--\blpage{1142}
(\byear{2022})
\end{barticle}
\endbibitem

\bibitem[\protect\citeauthoryear{Zhu et~al.}{2020}]{ZhuLiuZhou20}
\begin{barticle}
\bauthor{\bsnm{Zhu}, \binits{H.}},
\bauthor{\bsnm{Liu}, \binits{T.}},
\bauthor{\bsnm{Zhou}, \binits{E.}}:
\batitle{Risk quantification in stochastic simulation under input uncertainty}.
\bjtitle{ACM Trans. Model. Comput. Simul.}
\bvolume{30}(\bissue{1}),
\bfpage{1}
(\byear{2020})
\end{barticle}
\endbibitem

\end{thebibliography}

\begin{appendices}
\section{Proof of Main Theorem}
\subsection{Proof of Theorem \ref{thm:LSE_sieve}}
\begin{lemma}\label{lem:sub_Gauss_avg}
    Let $\bar{\sigma}_i$ be the sub-Gaussianilty of the average $\bar{\varepsilon}_i=\frac{1}{m}\sum_{j=1}^m\varepsilon_{ij}$ where $\varepsilon_{ij}$ are the zero-mean sub-Gaussian noise satisfying assumption \ref{assump:subG}. Then $$\max_{i=1,\cdots,n}\bar{\sigma_{i}}^2\leqslant C{\sigma}^2/m,$$
    where $C$ is some universal constant.
\end{lemma}
Lemma \ref{lem:sub_Gauss_avg}  is a direct result of the Bernstein inequality. 

\begin{lemma}\label{lem:empirical_produt_avg}
For any $\{h_i\in\Real\}_{i=1}^n$ and $\delta>0$, 
\[\pr\left(|\sum_{i=1}^n\bar{\varepsilon}_ih_i|\geqslant \delta\right)\leqslant 2\exp\left[-\frac{\delta^2 m}{C \sum_{i=1}^nh_i^2}\right]\]
    where $\bar{\varepsilon}_i=1/m\sum_{j=1}^m\varepsilon_{ij}$ are the averaged noise in Lemma \ref{lem:sub_Gauss_avg} and $C$ is some universal constant.
\end{lemma}
\begin{proof}
    From the sub-Gaussianilty of $\bar{\varepsilon}$ and Lemma~\ref{lem:sub_Gauss_avg}, we can show that for any $\omega$
    \[\mathit{E} \exp\left[\omega \sum_{i=1}^nh_i\bar{\varepsilon}_i\right] =\prod_{i=1}^n\mathit{E}\exp\left[\omega\bar{\varepsilon}_ih_i\right]\leqslant \exp\left[C_1\frac{\sigma^2\omega^2}{m}\sum_{i=1}^nh_i^2\right],\]
    where $C_1>0$ is the universal constant in Lemma~\ref{lem:sub_Gauss_avg}. We then use Chebyshev's inequality to show that for any $\omega>0$
    \[\pr\left(\sum_{i=1}^n\bar{\varepsilon}_ih_i\geqslant \delta\right)\leqslant \exp\left[C_1\frac{\sigma^2\omega^2}{m}\sum_{i=1}^nh_i^2-\delta\omega\right].\]
    Let $\omega=\delta m/(2C_1\sigma^2\sum_{i=1}^nh_i^2)$ to minimize the quantity $C_1\frac{\sigma^2\omega^2}{m}\sum_{i=1}^nh_i^2-\delta\omega$, we then have the following inequality independent of $\omega$
    \[\pr\left(\sum_{i=1}^n\bar{\varepsilon}_ih_i\geqslant \delta\right)\leqslant \exp\left[-\frac{\delta^2m}{4C_1\sigma^2\sum_{i=1}^nh_i^2}\right].\]
    Because the same inequality holds for $-\sum_{i=1}^n\bar{\varepsilon}_ih_i$, the lemma follows.
\end{proof}

The following theorem is adapted from Lemma 3.2 in  \cite{geer2000empirical}:
\begin{lemma}\label{lem:geer_lem3} Suppose $W_i$ are i.i.d. $\sigma$-sub-Gaussian r.v.s. Assume $\sup_{h\in\CalH}\|h\|_n\leqslant R$. There exist constants $C$ and $C_1$  depending only on $\sigma$ such that for all $\delta>0$ and $K>0$ satisfying
\[\sqrt{n}\delta\geqslant 2C\int^{R}_{\delta/(2^3K)}\sqrt{H(u,\CalH,\|\cdot\|_n)}\mathrm{d}u\vee R,\]
we have
\begin{align*}
    \pr\left(\{\sup_{h\in\CalH}|\langle W, h \rangle_n|\geqslant\delta\}\cap\{\|W\|_n\leqslant K\}\right)\leqslant C_1\exp\left[-\frac{n\delta^2}{1152\sigma^2R^2}\right].
\end{align*}
\end{lemma}

\begin{proposition}\label{prop:emp_process}
    Suppose $\sup_{h\in\CalH}\|h\|_n\leqslant R$. Then for some constants $C_1$, $C_2$ depending only on $\sigma$, and  for $\bar{\delta},\underline{\delta}>0$ satisfying $\bar{\delta}/\underline{\delta}<R$ and
    \[\sqrt{nm}\bar{\delta}\geqslant  2C_1\int^R_{\bar{\delta}/(2^3\underline{\delta})}\sqrt{H(u,\CalH,\|\cdot\|_n)}\mathrm{d}u\vee R,\]
    we have
    \[\pr\left(\{\sup_{h\in\CalH}|\langle \bar{\varepsilon},h\rangle_n|\geqslant \bar{\delta}\}\cap\{\|\bar{\varepsilon}\|_n^2\leqslant\underline{\delta}\} \right)\leqslant C_1 \exp\left[-\frac{mn{\bar{\delta}}^2}{C_2R^2}\right]\]
    where $\bar{\varepsilon}(\BFx_i)=1/m\sum_{j=1}^m\varepsilon_{ij}$ are the averaged noise in Lemma \ref{lem:sub_Gauss_avg}.
\end{proposition}
\begin{proof}
    Substitute Lemma~\ref{lem:empirical_produt_avg} for the sub-Gaussianality $\sigma$ in Lemma~\ref{lem:geer_lem3}, we can immediately have the result.
\end{proof}

\begin{proof}[Proof of Theorem \ref{thm:LSE_sieve}:]
    For simplicity, write $\hat{f}_{n,m}$ as $\hat{f}$, and write $f^*_{n,m}$ as $f^*$. Because $\hat{f}$ is the miniimizer of the problem $\min_{h\in\CalF_{n,m}}\|h-\bar{y} \|_n^2=\min_{h\in\CalF_{n,m}}\|h-f-\bar{\varepsilon} \|_n^2$, we can have $\|\hat{f}-f-\bar{\varepsilon}\|_n^2\leqslant \|{f}^*-f-\bar{\varepsilon}\|_n^2$. This inequality can be rewritten as 
    \begin{equation}\label{eq:thm1_pf_1}
        \|\hat{f}-f\|_n^2\leqslant 2\langle \bar{\varepsilon},\hat{f}-f^*\rangle_n+\|f^*-f\|_n^2.
    \end{equation}
    By triangle inequality, we can further have
    \begin{align}
        \|\hat{f}-f^*\|_n^2&\leqslant 2\|\hat{f}-f\|^2_n+2\|f^*-f\|^2_n\nonumber\\
        &\leqslant 4\langle \bar{\varepsilon},\hat{f}-f^*\rangle_n+4\|f^*-f\|_n^2\label{eq:thm1_pf_2}
    \end{align}
    where the second line is from \eqref{eq:thm1_pf_1}.

    If $\|f^*-f\|_n^2\geqslant \langle \bar{\varepsilon},\hat{f}-f^*\rangle_n $, \eqref{eq:LSE_convergence} can be derived directly from \eqref{eq:thm1_pf_2}.

     If $\|f^*-f\|_n^2\leqslant \langle \bar{\varepsilon},\hat{f}-f^*\rangle_n $, we have
     \begin{equation}\label{eq:thm1_pf_3}
         \|\hat{f}-f^*\|_n^2\leqslant 8  \langle \bar{\varepsilon},\hat{f}-f^*\rangle_n.
     \end{equation}
     By the Cauchy-Schwarz inequality,  we can deduce from \eqref{eq:thm1_pf_3} that $\|\hat{f}-f^*\|_n\leqslant 8\underline{\delta}$ on the event $\|\bar{\varepsilon}\|_n\leqslant \underline{\delta}$. On this event, it suffices to discuss the probability of the event $\|\hat{f}-f^*\|_n\leqslant \underline{\delta}$ for $\bar{\delta}\leqslant 8\underline{\delta}$. Let $S=\min\{s=0,1,\cdots:2^s\bar{\delta}\geqslant 8\underline{\delta}\}$. Then using the peeling technique, we have
     \begin{align}
         &\pr\left( \{\|\hat{f}-f^*\|_n^2\geqslant \bar{\delta}^2\}\cap\{\|\bar{\varepsilon}\|_n\leqslant\underline{\delta}\}\right)\\
         \leqslant{}&  \pr\left( \{\langle \bar{\varepsilon},\hat{f}-f^*\rangle_n\geqslant \bar{\delta}^2/8\}\cap\{\|\bar{\varepsilon}\|_n\leqslant\underline{\delta}\}\right)\nonumber\\
         ={}& \pr\left( \bigcup_{s= 0}^S\{2^{2s+2}\bar{\delta}^2/8\geqslant \langle \bar{\varepsilon},\hat{f}-f^*\rangle_n\geqslant 2^{2s}\bar{\delta}^2/8\}\cap\{\|\bar{\varepsilon}\|_n\leqslant\underline{\delta}\}\right)\nonumber\\
         \leqslant{}& \sum_{s=0}^S\pr\left(\sup_{h\in\CalF_n(2^{s+1}\bar{\delta}/\sqrt{8})}\langle \bar{\varepsilon},h-f^*\rangle_n\geqslant 2^{2s-3}\bar{\delta^2}\cap\{\|\bar{\varepsilon}\|_n\leqslant\underline{\delta}\}\right)\label{eq:thm1_pf_4}.
     \end{align}
Notice that if we set $\underline{\delta}=\sigma/\sqrt{m}$ and choose $\sqrt{nm}\delta_{n,m}^2\geqslant c\Psi(\delta_{n,m}) $, i.e., $\delta_{n,m}$ satisfies condition \eqref{eq:thm1_condition}, then for any $2^{s+1}{\delta}\geqslant \delta_{n,m}$ with $2^{s+1}\delta<2^7\sigma/\sqrt{m}$, condition \eqref{eq:thm1_condition} is also satisfied:
\[\sqrt{nm}2^{2s+2}\delta^2\geqslant c\Psi(2^{s+1}\delta)\]
for $\Psi(\delta)/\delta^2$ is a non-decreasing function of $\delta$ for $0<\delta<2^7\sigma/\sqrt{m}$. We then can apply Proposition~\ref{prop:emp_process} to each term of the summation \eqref{eq:thm1_pf_4}:
\begin{align*}
    \pr\left(\bigg\{\sup_{h\in\CalF_n(\frac{2^{s+1}\bar{\delta}}{\sqrt{8}})}\langle \bar{\varepsilon},h-f^*\rangle_n\geqslant 2^{2s-3}\bar{\delta}^2\bigg\}\cap\{\|\bar{\varepsilon}\|_n\leqslant\underline{\delta}\}\right)\leqslant C_1\exp\left[-\frac{mn2^{4s}\bar\delta^4}{C_22^{2s+2}\bar\delta^2}\right]
\end{align*}
where $C_1,C_2$ only depend on $\sigma$. So \eqref{eq:thm1_pf_4} becomes
\begin{align}
     \pr\left( \{\|\hat{f}-f^*\|_n\geqslant \delta_{n,m}\}\cap\{\|\bar{\varepsilon}\|_n\leqslant \frac{\sigma}{\sqrt{m}} \}\right)&\leqslant \sum_{s= 0}^S C\exp\left[-\frac{mn2^{2s}\delta_{n,m}^2}{4C}\right]\nonumber\\
     &\leqslant C\exp\left[-\frac{mn\delta_{n,m}^2}{C}\right] \label{eq:thm1_pf_5}.
\end{align}
where $C$  only depends on $\sigma$.  So 
\[ \pr\left( \{\|\hat{f}-f^*\|_n\geqslant \delta_{n,m}\right)\leqslant C\exp\left[-\frac{mn\delta_{n,m}^2}{C}\right]  +\pr\left(\|\bar{\varepsilon}\|_n\geqslant\frac{\sigma}{\sqrt{m}}\right). \]
\end{proof}

\subsection{Proof of Theorem \ref{thm:nested_convergence}}

\begin{lemma}\label{lem:sqrt_m_hat_f-f}
     If $m$ satisfies
    $\sqrt{m}(\delta_{n,m}+\|{f}^*-f\|_n)=\CalO(1)$, $\sqrt{m}\|\hat{f}-f^*\|_n=\CalO(1)$.
\end{lemma}
\begin{proof} From triangle inequality:
    \begin{align*}
        \sqrt{m}\|\hat{f}-f^*\|_n&\leqslant \sqrt{m}\left(\|\hat{f}-f\|_n+\|{f}^*-f\|_n\right)\\
        &\leqslant \sqrt{m}\left(\delta_{n,m}+2\|{f}^*-f\|_n\right)=\CalO(1).
    \end{align*}
\end{proof}
\begin{proof}[Proof of Theorem \ref{thm:nested_convergence}]
    By the triangle inequality, we have 
\begin{align}\label{eq:thm2_1}
    |\hat \theta_{n,m}-\theta| ={}& \left|\mathit{E} [\eta(f(X))] -\frac{1}{n}\sum_{i=1}^n \eta(f(\BFx_i))+ \frac{1}{n}\sum_{i=1}^n \eta(f(\BFx_i))- \frac{1}{n}\sum_{i=1}^n \eta(\hat{f}(\BFx_i))\right|\nonumber\\
    \leqslant{}& \underbrace{\left|\mathit{E} [\eta(f(X))] -\frac{1}{n}\sum_{i=1}^n \eta(f(\BFx_i))\right|}_{I_1}
    + \underbrace{ \left|\frac{1}{n}\sum_{i=1}^n \bigl[\eta(f(\BFx_i))- \eta(\hat{f}(\BFx_i)) \bigr] \right| }_{I_2}. 
\end{align}

Assumption \ref{assump:b_bounded} and conditions on $\eta$ imply that both $\|f\|_{L_\infty} $ and $\|\eta(f(\cdot))\|_{L_\infty}$ are finite. 
Therefore, $\eta(f(\BFx_i))$'s are bounded random variables, thereby being sub-Gaussian. 
Then, for a fixed value $T\gg 1$, the Bernstein inequality  implies that 
\begin{align}\label{eq:thm2_2}
    \pr(I_1 \geqslant Tn^{-1/2}))\leqslant C_1e^{-C_2T^2}
\end{align}
for some constants $C_1$ and $C_2$ independent of $n$.

It remains to bound $I_2$.  It follows from Taylor's expansion and the triangle inequality that
\begin{align*}
    I_2 ={}& \left|\frac{1}{n}\sum_{i=1}^n \eta'(f(\BFx_i))(f(\BFx_i)-\hat{f}(\BFx_i)) + \frac{1}{2n}\sum_{i=1}^n \eta''(\tilde{z}_i)(f(\BFx_i)-\hat{f}(\BFx_i))^2 \right|\\
    \leqslant{}& \underbrace{\left|\frac{1}{n}\sum_{i=1}^n \eta'(f(\BFx_i))(f(\BFx_i)-\hat{f}(\BFx_i))\right|}_{I_{21}} 
    + \underbrace{\left|\frac{1}{2n}\sum_{i=1}^n \eta''(\tilde{z}_i)(f(\BFx_i)-\hat{f}(\BFx_i))^2 \right|}_{I_{22}},
\end{align*}
where $\tilde{z}_i$ is a value between $f(\BFx_i)$ and $\hat{f}(\BFx_i)$.

For $I_{21}$, we use the Triangle inequality again to have
\begin{align}
    &  {\left|\frac{1}{n}\sum_{i=1}^n \eta'(f(\BFx_i))(f(\BFx_i)-\hat{f}(\BFx_i))\right|}\nonumber\\
    \leqslant{}&\left|\frac{1}{n}\sum_{i=1}^n \eta'(f(\BFx_i))(f(\BFx_i)-{f}^*(\BFx_i))\right| +\left|\frac{1}{n}\sum_{i=1}^n \eta'(f(\BFx_i))(f^*(\BFx_i)-\hat{f}(\BFx_i))\right| \nonumber\\
     \leqslant{}& \|\eta'\circ f\|_n\|f-f^*\|_n+\left|\frac{1}{n}\sum_{i=1}^n \eta'(f(\BFx_i))(f^*(\BFx_i)-\hat{f}(\BFx_i))\right|, \label{eq:thm2_pf_I21}
\end{align}
    where the first term in the last line is from Cauchy–Schwarz inequality. For the second term in the last line, we further decompose it as follows:
    \begin{align}
        & \left|\frac{1}{n}\sum_{i=1}^n \eta'(f(\BFx_i))(f^*(\BFx_i)-\hat{f}(\BFx_i))\right| \nonumber\\
        \leqslant {}& \left|\frac{1}{n}\sum_{i=1}^n \left(\eta'(f(\BFx_i))-\mathit{E} \eta'(f(\BFx_i)) \right)(f^*(\BFx_i)-\hat{f}(\BFx_i))\right|\label{eq:thm2_3}\\
        &\quad +\left|\frac{1}{n}\sum_{i=1}^n \mathit{E}\left[\eta'(f(\BFx_i))\right](f^*(\BFx_i)-\hat{f}(\BFx_i))\right| \label{eq:thm2_4}.
    \end{align}
    For \eqref{eq:thm2_4}, because we have assumed that any function in $\CalF$ and $\CalF_{n,m}$ is bounded by $b$ and $P_X$-Donsker and, in the later proofs for Theorem \ref{thm:risk_convergence}, we will show that $|f^*(\BFx_i)-\hat{f}(\BFx_i)|\leqslant\|\hat{f}-f^*\|_{L_\infty}=o_{\pr}(1)$. Therefore, it is straightforward to use the central limit theorem and Bernstein's inequality to get the following bound directly:
    \begin{align}
        &\left|\frac{1}{n}\sum_{i=1}^n \mathit{E}\left[\eta'(f(\BFx_i))\right](f^*(\BFx_i)-\hat{f}(\BFx_i))\right|\nonumber\\
        \leqslant& \left|\mathit{E}\left[\eta'(f(X))\right]\mathit{E}_{X,\{\varepsilon_i\}} [f^*(X)-\hat{f}(X)]\right|+Tn^{-\frac{1}{2}}\label{eq:thm2_5}.
    \end{align}
    with probability $1-C_3e^{-C_4 T^2}$ where constants $C_3$ and $C_4$  are independent of $n$.
    
    For \eqref{eq:thm2_3}, because $\eta'$ is bounded so $W_i=\eta'(f(\BFx_i))-\E \eta'(f(\BFx_i))$ is a zero-mean sub-Gaussian r.v. for each $i$.  Also, from Lemma \ref{lem:sqrt_m_hat_f-f}, we can notice that for any $m$ satisfying condition $\sqrt{m}(\delta_{n,m}+\|f-f^*\|_n)=\CalO(1)$, then $\|f^*-\hat{f}\|_n\leqslant R/\sqrt{m}$ for some constant $R$.  Let $\sigma_\eta$ denote the sub-Gaussianality of $W_i$. Then from Lemma~\ref{lem:geer_lem3}, we have for any $\delta$ satisfying $R/\sqrt{m}>\delta/\sigma_\eta$ and
    \begin{equation}\label{eq:thm2_condition_W}
        \sqrt{n}\delta\geqslant 2C\left(\int^{R/\sqrt{m}}_{\delta/(8\sigma_\eta)}\sqrt{H(u,\CalF_{n,m},\|\cdot\|_n)}\mathrm{d}u\vee \frac{R}{\sqrt{m}}\right)
    \end{equation}
    we have
    \begin{align*}
        \pr\left(\langle W,f^*-\hat{f}\rangle_n>\delta\}\right)
        \leqslant   C_3\exp\left[-\frac{mn\delta^2}{C_4R^2}
    \right]+\pr\left(\|W\|_n\leqslant \sigma_\eta\right).
    \end{align*}
Recall that $\delta_{n,m}$ satisfies $\sqrt{m}\delta_{n,m}=\CalO(1)$ and
\[\sqrt{mn}\delta^2_{n,m}\geqslant c\left(\int_{\frac{\delta_{n,m}^2\sqrt{m}}{2^7{\sigma}}}^{{\delta_{n,m}}}\sqrt{H(u,\CalF_n({\delta}_{n,m}),\|\cdot\|_n)}\mathrm{d}u \vee {\delta}_{n,m}\right).\]
Let $\delta=\delta_{n,m}^2$, then $\delta$ satisfies condition \eqref{eq:thm2_condition_W}. So 
\begin{equation}\label{eq:thm2_6}
    \left|\frac{1}{n}\sum_{i=1}^n \left(\eta'(f(\BFx_i))-\E \eta'(f(\BFx_i)) \right)(f^*(\BFx_i)-\hat{f}(\BFx_i))\right|\geqslant T\delta_{n,m}^2.
\end{equation}
with probability $C_5e^{-C_6T^2}$  where constants $C_5$ and $C_6$ are independent of $n$ and $m$.


    For $I_{22}$, from the boundedness of $\eta''$, it is straightforward to derive that
    \begin{equation}
        I_{22}\leqslant \frac{\|\eta''\|_{L_\infty}}{2}\|f-\hat{f}\|_n^2\leqslant T\delta^2_{n,m}+\|f-f^*\|_n^2 \label{eq:thm2_7}.
    \end{equation}
  with probability $1-C_5e^{-C_6T^2}$
    Substitute \eqref{eq:thm2_pf_I21}, \eqref{eq:thm2_5}, \eqref{eq:thm2_6}, and \eqref{eq:thm2_7} into $I_2$, together with \eqref{eq:thm2_2} for $I_1$, we can have 
    \begin{equation}
        |\hat \theta_{n,m}-\theta| \leqslant T\left(n^{-\frac{1}{2}}+\delta_{n,m}^2+\|f^*_{n,m}-f\|_n+\hat{\CalL}\right),
    \end{equation}
    with probability $1-C_7e^{-C_8T^2}$ where constants $C_7$ and $C_8$ are independent of $n$ and $m$.

    Define $\Delta_{n,m}=n^{-\frac{1}{2}}+\delta_{n,m}^2+\|f^*_{n,m}-f\|_n+\hat{\CalL}$. To prove the convergence in $l^p$ for $p<\infty$, notice 
    \begin{align*}
        \left(\mathit{E} \left|\frac{\hat \theta_{n,m}-\theta}{\Delta_{n,m}}\right|^p\right)^{1/p}=&\left(\int_0^\infty\pr\left(\left|\frac{\hat \theta_{n,m}-\theta}{\Delta_{n,m}}\right|^p\geqslant T\right) \mathrm{d}T\right)^{1/p}\\
        =& \left(\int_0^\infty pT^{P-1}\pr\left(\frac{\hat \theta_{n,m}-\theta}{\Delta_{n,m}}\geqslant T\right) \mathrm{d}T\right)^{1/p}\\
        =& \left(\int_0^\infty pT^{P-1}C_7e^{-C_8 T^2} \mathrm{d}T\right)^{1/p}=C_p<\infty.
    \end{align*}
    Therefore, we have $ \left(\mathit{E} \left| {\hat \theta_{n,m}-\theta}\right|^p\right)^{1/p}\leqslant C_p \Delta_{n,m}$, which finishes the proof.
\end{proof}

\subsection{Proof of Theorem \ref{thm:risk_convergence}}
We first define two quantities known as the \emph{empirical Rademacher complexity} and \emph{$L_2$-Rademacher complexity} of a space $\CalH$, respectively:
\begin{align*}
    &\hat{\CalR}_n(\delta,\CalH)=\mathit{E}_{r}\left[\sup_{h\in\CalH\cap B(\delta,\|\cdot\|_n)} \langle r,h\rangle_n\right],\\
    &{\CalR}_n(\delta,\CalH)=\mathit{E}_{\{\BFx_i\},r}\left[\sup_{h\in\CalH\cap B(\delta,\|\cdot\|_{L_2})} \langle r,h\rangle_n\right]
\end{align*}
where $B(\delta,\|\cdot\|)=\{h: \|h\|\leqslant \delta\}$ is the ball with radius $\delta$ under norm $\|\cdot\|$ and $r(\BFx_i)=r_i$ are i.i.d. Rademacher r.v.s, i.e., $\pr(r_i=1)=\pr(r_i=-1)=1/2$. According to the Dudley's entropy integral bound, 
\begin{equation}\label{eq:Dudley}
    \hat{\CalR}_n(\delta,\CalH)\leqslant \inf_{\varepsilon\geqslant 0}\left\{4\varepsilon+\frac{C}{\sqrt{n}}\int_\varepsilon^{\delta} \sqrt{H(u,\CalH\cap B(\delta,\|\cdot\|_n),\|\cdot\|_n)}\mathrm{d}u\right\},
\end{equation}
where $C>0$ is some universal constant. Please refer to \cite[Chapter 5]{Wainwright19} for details.

We now can state the following theorem regarding the relationships between the empirical and $L_2$ norms/Rademacher complexities. This theorem is adapted from Theorem 14.1 and Proposition 14.25 of \cite{Wainwright19}:
\begin{theorem}\label{thm:wainwright}
Suppose $\sup_{h\in\CalH}\|h\|_{\infty}\leqslant b$ and $\CalH$ is a star-shaped. Let 
\begin{align}\label{eq:thm2_wainwright}
    &\hat{\delta}_n=\inf\{\delta: \hat{\CalR}_n(\delta,\CalH)\leqslant \delta^2/b\},\ {\delta}_n=\inf\{\delta: {\CalR}_n(\delta,\CalH)\leqslant \delta^2/b\}.
\end{align}
Assume $\log\log 1/\delta=\CalO(n\delta_n^2)$. Then there exist constants $m_1$,$m_2$, $c_1$ and $a$ such that with probability at least $1-ae^{-n\delta_n^2/(c_0b)}$:
\begin{align*}
    &m_1\delta_n\leqslant \hat{\delta}_n\leqslant m_2\delta_n,\\
    &\sup_{h\in\CalH}\left|\|h\|_n-\|h\|_{L_2}\right|\leqslant c_1\delta_n.
\end{align*}
\end{theorem}

Remind that $\CalF_{n,m}$ and $\CalF$ satisfy Assumptions \ref{assump:b_bounded} and \ref{assump:L_infinity_embedding}. Using Theorem \ref{thm:wainwright}, we can immediately derive the following bound regarding the $L_2$ and empirical norms for any function in $\CalF\cup\CalF_{n,m}$:

\begin{lemma}\label{lem:h_L2_norm}
There exist a universal constant $C$ such that for $\delta>0$ satisfying:
    \begin{equation}\label{eq:thm3_coro_condition}
        \frac{\delta^2}{b}\geqslant \inf_{\varepsilon\geqslant 0}\left\{4\varepsilon+\frac{C}{\sqrt{n}}\int_\varepsilon^{\delta} \sqrt{H(u,\CalH\cap B(\delta,\|\cdot\|_n),\|\cdot\|_n)}\mathrm{d}u\right\},
    \end{equation}
    then
    \[\sup_{h\in\CalF\cup\CalF_{n,m}}\|h\|_{L_2}\leqslant \sup_{h\in\CalF\cup\CalF_{n,m}}\|h\|_n+\frac{\delta}{\sqrt{n}} \]
    with probability at least $1-ae^{-\delta^2/(c_0b)}$.
\end{lemma}
\begin{proof}
    According to Assumption \ref{assump:b_bounded}, $H(u,\CalF_{n,m},\|\cdot\|_n)\leqslant H(u,\CalF,\|\cdot\|_n)$. So it is straightforward to derive that $H(u,\CalF\cup\CalF_{n,m},\|\cdot\|_n) \leqslant 2H(u,\CalF,\|\cdot\|_n)$ according to the definition of $u$-entropy.  Using Dudley's integral \eqref{eq:Dudley}, we have
    \begin{align*}
        \hat{\CalR}_n(\delta,\CalF\cup\CalF_{n,m})&\leqslant \inf_{\varepsilon\geqslant 0}\left\{4\varepsilon+\frac{C}{\sqrt{n}}\int_\varepsilon^{\delta} \sqrt{H(u,(\CalF\cup\CalF_{n,m})\cap B(\delta,\|\cdot\|_n),\|\cdot\|_n)}\mathrm{d}u\right\}
    \end{align*}
    where $C$ is some universal constant. Therefore, for any $\delta$ satisfying the condition \eqref{eq:thm3_coro_condition}, $\delta$ must be the upper bound of $\hat{\delta}_n$ in \eqref{eq:thm2_wainwright} up to a multiplicative constant. Then according to Assumption \ref{assump:L_infinity_embedding} and Theorem \ref{thm:wainwright}, we have
    \begin{align*}
        \sup_{h\in\CalF\cup\CalF_{n,m}}\|h\|_{L_2}&= \sup_{h\in\CalF\cup\CalF_{n,m}}\|h\|_{n}+\|h\|_{L_2}-\|h\|_{n}\\
        &\leqslant \sup_{h\in\CalF\cup\CalF_{n,m}}\|h\|_{n}+ \sup_{h\in\CalF\cup\CalF_{n,m}}\left||h\|_{L_2}-\|h\|_{n}\right|\\
        &\leqslant \sup_{h\in\CalF\cup\CalF_{n,m}}\|h\|_n+\frac{\delta}{\sqrt{n}}
    \end{align*}
    with probability at least $1-ae^{-\delta^2/(c_0b)}$.
\end{proof}

\begin{lemma}\label{lem:f_hat_f_L_infinity}
    Suppose Assumption~\ref{assump:support}--\ref{assump:L_infinity_embedding} holds. Let $\delta_{n,m}$, $f^*$, and $\hat{f}$ be as defined in Theorem \ref{thm:LSE_sieve}. Then with probability $1-C_1e^{-C2 T^2}$
    \[\rho_n:=\|\hat{f}-f^*\|_{L_\infty}\leqslant T\left(\left|\delta_{n,m}+\|f^*-f\|_n+{\delta}/{\sqrt{m}}\right|^\alpha\right)\]
    where $\delta$ satisfies condition \eqref{eq:thm3_coro_condition}, $m$ satisfies $\sqrt{m}(\delta_{n,m}+\|f^*-f\|_n)=\CalO(1)$, and $C_1$ and $C_2$ are constants independent of $n$ and $m$.
\end{lemma}
Lemma~\ref{lem:f_hat_f_L_infinity} is a direct result of Theorem \ref{thm:LSE_sieve}, Assumption~\ref{assump:L_infinity_embedding}, and Lemma~\ref{lem:h_L2_norm}.  

\begin{proof}[Proof of Theorem \ref{thm:risk_convergence}] 
The proof of Theorem \ref{thm:risk_convergence} is merely repeating the proof of Theorem 4 of \cite{wang2024smooth}. The only difference is that the term $l_n$ in the proof of Theorem 4 of \cite{wang2024smooth} is
\begin{align*}
    l_n = \left|\delta_{n,m}+\|f^*-f\|_n+{\delta}/{\sqrt{m}}\right|^\alpha
\end{align*}
in our case.
\end{proof}

\section{Generation of Random Correlation Matrix}\label{append:sigma_generation} 

The correlation matrix generation follows the standard procedure outlined in \cite{wang2024smooth}.  We form the lower-triangular matrix \([\sigma_{ij}]_{i,j=1}^q\) by applying the Cholesky decomposition to a covariance matrix, \(\Sigma\). This matrix \(\Sigma\) is produced through the following steps. Initially, we create a random correlation matrix \(V\) using the `gallery` function from the Matlab library. This function's inputs are the eigenvalues of \(V\), which are derived from a uniform distribution, Uniform\([0.1, 0.8]\), for \(i=1,\cdots,q\), and subsequently normalized to ensure that the sum of the eigenvalues, \(\sum_i \lambda_i\), equals \(q\). This normalization is crucial to align the sum of the eigenvalues with the trace of \(V\), necessitating that all diagonal elements of \(V\) be ones. Following this, we define \(\Sigma = DVD\), where \(D\) is a diagonal matrix whose elements are selected from a uniform distribution, Uniform\([0.1, 0.8]\).

\section{Background Knowledge}\label{append:back_know} 
In this subsection, we provide an introduction to Sobolev spaces and the Donsker class for a comprehensive treatment of the subject. 

\subsection{Sobolev Spaces}
Before introducing the concept of Sobolev spaces, it is necessary to define the concept of a weak derivative.

\textbf{Notation.} Let $\mathcal{C}^{\infty}_{c}(\Omega)$ denote the space of infinitely differentialble functions. Furthermore, let $L_{loc}^1(\Omega)$ denote a set of all locally integrable functions on $\Omega$.

\begin{definition}\label{def:weak_derivative}(Section 5.2.1 in \cite{evans2022partial})
 Suppose $u$, $v\in L_{loc}^1(\Omega)$, and $\alpha$ is a multi-index. We say that $v$ is the $ \boldsymbol{\alpha}^{th}$ -weak partial derivative of $u$, written $D^{ \boldsymbol{\alpha}}u=v$, provided that 
 $$\int_{\Omega}uD^{ \boldsymbol{\alpha}}\phi dx=(-1)^{| \boldsymbol{\alpha}|}\int_{\Omega}v\phi \mathrm{d}x$$
 for all test functions $\phi\in\mathcal{C}^{\infty}_{c}(\Omega)$, with compact support in $\Omega$. Here      $  \boldsymbol{\alpha}=(\alpha_1,\cdots,\alpha_d)\in\NatInt^d$ consists of $d$ integers and 
\[D^{ \boldsymbol{\alpha}}f=\frac{\partial^{| \boldsymbol{\alpha}|}f}{\partial x_1^{\alpha_1}\cdots\partial x_d^{\alpha_d}}.\] \end{definition}

The intuition is that functions with weak derivatives appear differentiable everywhere except on sets of zero measure. The disregard of sets of zero measure by weak derivatives stems from their integral-based definition; integrals inherently overlook any behavior on sets of zero measure. For example

\begin{example}
  The characteristic function of rational numbers denoted as $f(t) = \mathbbm{1}_{\mathbb{Q}}(t)$, is an example of a function that is nowhere differentiable but possesses a weak derivative. The Lebesgue measure of the $f(t)$ is zero, hence the weak derivative of this function is $v(t)=0$.
\end{example}
 Thus, the weak derivative broadens the concept of differentiable by encompassing functions that are undifferentiable on sets of measure zero. We can now define a Sobolev space using the $L_{p}$ norm of the weak derivative.  
\begin{definition}\label{def:sobolev}(Section 5.2.2 in \cite{evans2022partial})The Sobolev space $W^{k,p}(\Omega)$ comprises all functions $u: \Omega \rightarrow \mathbb{R}$ that are locally summable and for which, given a multi-index $\alpha$ with $|\alpha| \leqslant k$, the $ \boldsymbol{\alpha}^{\text{th}}$ weak derivative $D^{ \boldsymbol{\alpha}}u$ exists in the weak sense belongs to $L_p(\Omega)$.
\end{definition}
The parameter $k$ in a Sobolev space $W^{k,p}$ is the smoothness of the Sobolev space. A higher $k$ means the functions are smoother. The parameter $p$
 affects the type of functions in the space, based on how they integrate. For instance, $p=\infty$ corresponds to functions whose derivatives up to order $k$ are essentially bounded on $\Omega$.  

\subsection{Donsker Class}
The following is the definition of Donsker class functions. 
\begin{definition}\label{def:donsker} (Section 2.8.2 in \cite{van1997weak})
A class of measurable functions $\CalF$ $f:\Omega\mapsto \Real$  is Donsker for a probability measure $P$ on $(\Omega, \mathcal{A})$ if the empirical process $\mathbb{G}_{n,P}=\sqrt{n}(P_n-P)$ satisfies the following : 
\begin{align*}
    \sup_{h\in BL_{1}}|\mathit{E}_{P} h(\mathbb{G}_{n,P})-\mathit{E}h(\mathbb{G}_{P})|\rightarrow 0
\end{align*}
where $\mathbb{G}_{\BFP}$  is a Brownian bridge, $\mathbb{P}_n$ is the empirical mearue based on i.i.d sample $X_1, \cdots, X_n\sim P$. And $BL_{1}$ is the set of all functions $h : l^{\infty} (\CalF)\mapsto \Real$ which is uniformly bounded by 1 and satisfy $|h(z_1)-h(z_2)|\leqslant \|z_1-z_2\|_{\CalF}.$   
\end{definition}

 The Donsker class is essentially a set of functions for which the empirical process converges in distribution to a Gaussian process in an infinite-dimensional space. More formally, let \(X_1, X_2, \cdots, X_n\) be independent and identically distributed random variables drawn from some distribution, and consider a class of measurable functions \(\mathcal{F}\) where each \(f \in \mathcal{F}\) maps from the space of the \(X_i\) to \(\mathbb{R}\). The empirical measure \(P_n\) is defined as \(P_n f = \frac{1}{n} \sum_{i=1}^n f(X_i)\) for any \(f \in \mathcal{F}\), and \(P\) is the true measure such that \(Pf = E[f(X)]\) for any \(f \in \mathcal{F}\). Using this definition, we can define a stochastic process \(\sqrt{n}(P_n - P)\), which is indexed by the whole function space \(\mathcal{F}\),  and we call this process \textit{empirical process}. The class \(\mathcal{F}\) is a Donsker class if the empirical processes converge in distribution to a Gaussian process \(G_P\) in the space \(L_\infty(\mathcal{F})\), the space of all bounded functions from \(\mathcal{F}\) to \(\mathbb{R}\), endowed with the uniform norm. This means that for large samples, the behavior of the empirical process can be approximated by that of a Gaussian process, which allows for the derivation of asymptotic distributions of various statistics and for constructing confidence sets, among other applications.

The concept of a Donsker class is powerful because it provides conditions under which complex, potentially infinite-dimensional objects (like empirical processes) exhibit predictable, well-understood asymptotic behavior. Identifying whether a function class is Donsker can be crucial for understanding the properties of statistical procedures, such as the consistency and asymptotic normality of estimators. For more details of the Donsker class, please refer to \cite{geer2000empirical,van1997weak}.
\end{appendices}

\end{document}